\documentclass[11pt,reqno]{amsproc}
\usepackage{lineno,hyperref}
\modulolinenumbers[5]
\usepackage{amsmath}
\usepackage{amssymb}
\usepackage{mathrsfs}
\usepackage{amsthm}
\usepackage{bm}
\usepackage{bbm}
\usepackage{graphicx}
\usepackage{booktabs}
\usepackage{float}
\usepackage{subfig}
\usepackage{geometry}
\usepackage{color}
\newtheorem{Def}{Definition}[section]
\newtheorem{lem}[Def]{Lemma}
\newtheorem{tho}[Def]{Theorem}
\newtheorem{prop}[Def]{Proposition}
\newtheorem{rem}[Def]{Remark}
\newtheorem{cor}[Def]{Corollary}
\newtheorem{Asp}{Assumption}

\newcommand{\R}{\mathbb R}
\newcommand{\ud}{\mathrm d}
\newcommand{\D}{\mathbb D}
\newcommand{\OO}{\mathcal O}

\newcommand{\cH}{\mathfrak H}
\newcommand{\E}{\mathbb E}
\newcommand{\GG}{\mathbb G}
\newcommand{\C}{\mathcal C}

\newcommand{\Z}{\mathbb Z}
\numberwithin{equation}{section}
\allowdisplaybreaks[4]

\begin{document}

\title[Convergence analysis of FDM]{
Convergence analysis of a finite difference method for stochastic Cahn--Hilliard equation}

\subjclass[2010]{65C30, 60H35, 60H15, 60H07}
\author{Jialin Hong}
\address{LSEC, ICMSEC, Academy of Mathematics and Systems Science, Chinese Academy of Sciences, and School of Mathematical Sciences, University of Chinese Academy of Sciences, Beijing 100049, China}
\email{hjl@lsec.cc.ac.cn}

\author{Diancong Jin}
\address{School of Mathematics and Statistics, Huazhong University of Science and Technology,  and Hubei Key Laboratory of Engineering Modeling and Scientific Computing, Huazhong University of Science and Technology, Wuhan 430074, China}
\email{jindc@hust.edu.cn}

\author{Derui Sheng}
\address{LSEC, ICMSEC, Academy of Mathematics and Systems Science, Chinese Academy of Sciences, and School of Mathematical Sciences, University of Chinese Academy of Sciences, Beijing 100049, China}
\email{sdr@lsec.cc.ac.cn}

\thanks{This work is supported by the National key R\&D Program of China under Grant No.\ 2020YFA0713701, National Natural Science Foundation of China (Nos.\ 11971470, 11871068, 12031020, 12022118), and the Fundamental Research Funds for the Central Universities 3004011142.}

\keywords{strong convergence rate, finite difference method, exponential Euler method, stochastic Cahn–Hilliard equation}

\begin{abstract}
This paper presents the convergence analysis of the spatial finite difference method (FDM) for the stochastic Cahn--Hilliard equation with Lipschitz nonlinearity and multiplicative noise. 
Based on fine estimates of the discrete Green function,
we prove that both the spatial semi-discrete numerical solution and its Malliavin derivative have strong convergence order $1$.
Further, by showing the negative moment estimates of the exact solution,
we obtain that the density of the spatial semi-discrete numerical solution converges in $L^1(\R)$ to the exact one. 
Finally, we apply an exponential Euler method to discretize the spatial semi-discrete numerical solution in time and show that the temporal strong convergence order is nearly $\frac38$,
 where a difficulty we overcome is to derive the optimal H\"older continuity of the spatial semi-discrete numerical solution.

%
%
%
\end{abstract}

\maketitle
\section{Introduction}

Consider 
the following stochastic Cahn--Hilliard equation
\begin{gather}\label{CH}
\begin{split}
\partial_t u+\Delta^2u&=\Delta f(u)+\sigma(u)\dot{W},\quad \text{in}~[0,T]\times\OO\\
\end{split}
\end{gather}
with initial condition $u(0,\cdot)=u_0$ and homogeneous Dirichlet boundary conditions (DBCs)
$u=\Delta u=0$ on $\partial\OO$. Here, $\OO:=(0,\pi)$, $T>0$, and $\{W(t,x),(t,x)\in[0,T]\times\OO\}$ is a Brownian sheet defined on some probability space $(\Omega,\mathscr F, \mathbb P)$. 
Assume that $u_0:\OO\rightarrow \R$ is a deterministic continuous function, and $\sigma$ is bounded and globally Lipschitz continuous.  Eq.\ \eqref{CH} is a well-known phenomenological model to  describe the complicated phase separation.
In the original form, $f$ is the derivative of the homogeneous free energy $F$ which contains a logarithmic term and in some cases  can be approximated by an even-degree polynomial with a positive dominant coefficient \cite{CC01}, for example, $F(x)=\frac{1}{4}x^4-\frac{1}{2}x^2$.
In this case, the truncation technique is usually used to localize \eqref{CH} such that the sequence $\{u_R\}_{R\ge1}$ given by
\begin{gather}\label{LCH}
\begin{split}
\partial_t u_R+\Delta^2u_R&=\Delta f_R(u_R)+\sigma(u)\dot{W},\quad \text{in}~[0,T]\times\OO\\
\end{split}
\end{gather}
can approximate $u$ in some sense  (see e.g., \cite{CC01,CH20}), where $\{f_R\}_{R\ge1}$ satisfies the global Lipschitz condition. 
Hence, an effective numerical method applied to Eq.\ \eqref{LCH} is expected to approximate  Eq.\ \eqref{CH} well.   
With this consideration, the present work investigates numerical methods for stochastic Cahn--Hilliard equations with Lipschitz nonlinearity (i.e., $f$ is globally Lipschitz continuous), which includes the linearized Cahn--Hilliard equation ($f=0$).

 The existing results on the numerical methods for stochastic Cahn--Hilliard equations mainly focus on the strong convergence analysis. Without being too exhaustive, 
we mention \cite{CCZZ18,LM11} on the finite element approximation for the case of $f=0$. 
For the case of polynomial nonlinearity and additive noise, \cite{KLM11} and \cite{FKLL18} respectively obtain the strong convergence of a spatial semi-discretization and a full discretization; 
\cite{CHS21,QW20} establish the strong convergence rates of full discretizations based on the finite element method and spectral Galerkin method in space, respectively. 
Concerning the case of multiplicative noise, \cite{CH20} presents the sharp strong convergence rate for a full discretization by using the spectral Galerkin method in space. 
In addition to the strong convergence analysis, the convergence analysis of densities of numerical solutions is also meaningful, 
which provides a theoretical foundation to approximate the density of the exact solution of the original system by means of  numerical methods.
There have been plenty results on density convergence of numerical solutions for various stochastic systems (see e.g., \cite{BT96,CCHS20,CHS19,HW96,KHA97,NP13}), of which we have not yet found relevant results for stochastic Cahn-Hilliard equations.
The present paper aims to approximate the density of the exact solution of Eq.\ \eqref{CH} via a spatial finite difference method (FDM) and present the strong convergence rate and density convergence of the associated numerical solution.

The spatial FDM has been employed to numerically solve, for instance, stochastic heat equations \cite{GI98,DG01,ACQ20} and stochastic wave equations \cite{CQ16}. 
First, we give subtle error estimates  between the discrete Green function of the spatial FDM  and the exact one. Then under the globally Lipschitz condition on $f$,
 we obtain the strong convergence order $1$
for the spatial semi-discrete numerical solution $u^n(t,x)$ based on the FDM, with $\frac{\pi}{n}$ being the spatial stepsize.  
 Further, it is shown that the Malliavin derivative of $u^n(t,x)$ has the strong convergence order $1$ as well. 
Combining the above results with the negative moment estimates of the exact solution,
we deduce that the spatial semi-discrete numerical solution admits a density, which converges in $L^1(\R)$ to the density of the exact solution.

For more effective computation, we further discretize $u^n$ via an exponential Euler method in time and obtain the full discretization $u^{m,n}=\{u^{m,n}(t,x),(t,x)\in[0,T]\times\OO\}$, where $\frac{T}{m}$ denotes the temporal stepsize.
As an explicit method, the exponential Euler method is more computationally efficient than the implicit method and does not suffer from the CFL condition. 
By investigating the temporal H\"older continuity of the spatial semi-discrete numerical solution,  we attain the strong convergence rate of the proposed method for Eq.\ \eqref{CH} with Lipschitz nonlinearity, namely 
\begin{equation}\label{Full-error}
	\|u^{m,n}(t,x)-u(t,x)\|_{L^p(\Omega)}
	\le C(\epsilon)(n^{-1}+m^{-\frac{3}{8}+\epsilon}),
\end{equation}
where $0<\epsilon\ll 1$.
The spatial convergence order $1$ and temporal convergence order nearly $\frac{3}{8}$ in \eqref{Full-error} are optimal in the sense that they coincide with the mean-square spatial and temporal H\"older continuity exponents of the exact solution, respectively. 
On the basis of \eqref{Full-error}, a localized argument leads to an $L^p(\Omega;\R)$ convergence order localized on a set of arbitrarily large probability for Eq.\ \eqref{CH} with polynomial nonlinearity. With the independent interest, when $f$ is a polynomial of degree $3$ with a positive dominant coefficient, we also establish the H\"older continuity and the uniform moment estimate of the exact solution. These are prepared for the density convergence analysis of the numerical solution of the  spatial FDM for Eq.\ \eqref{CH} with polynomial nonlinearity in our future work.


The rest of this paper is organized as follows. Section \ref{S2} gives the H\"older continuity and the uniform moment estimate of the exact solution. Then
we introduce the spatial FDM and study its strong convergence order in Section \ref{S3}. Section \ref{S4} presents the density convergence of the spatial semi-discrete numerical solution. In Section \ref{S5}, we further apply an exponential Euler method to obtain a full discretization and 
obtain its strong convergence order.

\section{Preliminaries}\label{S2}

%
%

Let $\C^\alpha(\OO)$ be the space of $\alpha$-H\"older continuous functions on $\OO$ if $\alpha\in(0,1)$, and  be the space of $\alpha$ times continuously differentiable functions on $\OO$ if $\alpha\in\mathbb N$. 
For $d\ge1$, we denote by $\|\cdot\|$ and $\langle\cdot,\cdot\rangle$ the Euclidean norm and inner product of $\R^d$, respectively. For $1\le q\le\infty$, we denote by 
$\|\cdot\|_{L^q}$  the usual norm of the space  $L^q(\OO):=L^q(\OO;\R)$. For $1\le p< \infty$ and a Banach space $(H,\|\cdot\|_H)$, let $L^p(\Omega;H)$ be the space of $H$-valued random variables with bounded $p$th moment, endowed with the norm $\|\cdot\|_{L^p(\Omega;H)}:=\left(\E[\|\cdot\|_H^p]\right)^{\frac{1}{p}}.$ 
Especially, we write $\|\cdot\|_p:=\|\cdot\|_{L^p(\Omega;\R)}$ for short. 
Hereafter, we use $C$ to denote a generic positive constant that may change from one place to another and depend on several parameters but never on the space and time stepsizes.  
Without illustrated, the supremum with respect to $t\in[0,T]$ (respectively, $x\in\OO$ and $(t,x)\in[0,T]\times\OO$) is denoted by $\sup_{t}$ (respectively, $\sup_{x}$ and $\sup_{t,x}$).
In this section, we present the regularity estimate of the exact solution of Eq.\ \eqref{CH} under Assumption \ref{A1} or \ref{A2}.

\begin{Asp}\label{A1}
$f$ satisfies the globally Lipschitz condition, i.e., there is  $K>0$ such that $$|f(y)-f(z)|\le K|y-z|,\quad\forall~y,z\in\R.$$
\end{Asp}
\begin{Asp}\label{A2}
$f$ is a polynomial of degree $3$ with a positive dominant coefficient, i.e.,
$f(x)=a_0x^3+a_1x^2+a_2x+a_3$ with $a_0>0$.
\end{Asp}

The physical importance of the Dirichlet problem is pointed out to us by M. E. Gurtin: it governs the propagation of a solidification front into an ambient medium which is at rest relative to the front \cite{DN91}; see \cite{EL92,CH20} and references therein for the study of Cahn--Hilliard equation with DBCs. In this case,
the Green function associated to $\partial_t+\Delta^2$  is  given by
$G_t(x,y)=\sum_{j=1}^\infty e^{-\lambda_j^2t}\phi_j(x)\phi_j(y),$ $t\in[0,T]$, $x,y\in\mathcal O$,
where $\lambda_j=-j^2$, $\phi_j(x)=\sqrt{2/\pi}\sin(jx),$ $j\ge1$. 
It is known that $\{\phi_j,j\ge1\}$ forms an orthonormal basis of $L^2(\mathcal O)$.
Denote $\GG_tv(x):=\int_{\mathcal O}G_t(x,y)v(y)\ud y$, $v\in \C(\OO).$ 
Similar to \cite[Lemma 1.2]{CC01},  there exist $C, c>0$ such that 
\begin{align}\label{Gtxy}
|G_t(x,y)|\le \frac{C}{t^{1/4}}\exp\Big(-c\frac{|x-y|^{4/3}}{|t|^{1/3}}\Big),\\\label{DGtxy}
|\Delta G_t(x,y)|\le  \frac{C}{t^{3/4}}\exp\Big(-c\frac{|x-y|^{4/3}}{|t|^{1/3}}\Big).
\end{align}

The well-posedness of the
stochastic Cahn--Hilliard equation under Assumption \ref{A2} with NBCs has been  established in \cite{CC01}. 
Since the Green function with DBCs and NBCs share similar properties, the existence and uniqueness of the solution to Eq.\ \eqref{CH} under Assumption \ref{A2}
can be obtained in an almost same way, and we present an outline of the idea here.
For $R\ge1$, let $K_R:\R\rightarrow\R$ be an even smooth cut-off function satisfying
\begin{align}\label{KR}
K_R(x)=1,\quad \text{if}~|x|<R;\qquad K_R(x)=0,\quad \text{if}~|x|\ge R+1,
\end{align}
and $|K_R|\le 1$, $|K_R^\prime|\le 2$. 
Consider a sequence of SPDEs
\begin{align}\label{eq.truncatedCH}
\partial_t \bar u_R(t,x)+\Delta^2\bar u_R(t,x)=\Delta \big(K_R(\|\bar{u}_R(t,\cdot)\|_{L^q})f(\bar u_R(t,x))\big)+\sigma(\bar u_R(t,x))\dot{W}(t,x)
\end{align}
with DBCs and $\bar u_R(0,\cdot)=u_0\in L^q(\OO)$ for some $q\ge4$. Define the stopping times
$$\tau_R:=\inf\{t\ge 0:\|\bar{u}_R(t,\cdot)\|_{L^q}\ge R\},\quad R\ge1.$$
Using the uniqueness of the solution of Eq.\ \eqref{eq.truncatedCH}, it is concluded from the local property of stochastic integrals that for $R^\prime>R$, $\bar{u}_{R^\prime}(t,\cdot)=\bar{u}_R(t,\cdot)$ for $t\le \tau_R$, so that a process $u$ can be defined by $u(t,\cdot)=\bar{u}_R(t,\cdot)$ for $t\le \tau_R$.
Set $\tau_\infty=\lim_{R\rightarrow\infty}\tau_R$. Then $u$ is the unique solution of \eqref{CH} on the interval $[0,\tau_\infty)$.  Further, $\{\bar{u}_R\}_{R\ge 1}$ are $\{\mathcal F_t\}_{t\in[0,T]}$-adapted stochastic processes such that for $\rho\in[q,\infty)$,
\begin{equation}\label{truncated-bound}
\sup_{R\ge1}\mathbb E\Big[\sup_{t}\|\bar{u}_R(t,\cdot)\|_{L^q}^\rho\Big]\le C(T,\rho,q)
\end{equation}
(see the second inequality in P794 of \cite{CC01}).
Based on \eqref{truncated-bound}, 
$\tau_\infty=+\infty$ a.s. (see  \cite[(2.36)]{CC01}), and thus under Assumption \ref{A2},
Eq.\ \eqref{CH} admits a global solution, i.e.,
\begin{align*}
u(t,x)&=\GG_tu_0(x)+\int_0^t\int_\OO \Delta G_{t-s}(x,y)f(u(s,y))\ud y\ud s\\
&\quad+\int_0^t\int_\OO G_{t-s}(x,y)\sigma(u(s,y))W(\ud s,\ud y),\quad (t,x)\in[0,T]\times\OO.
\end{align*}
It follows from Fatou's lemma and \eqref{truncated-bound} that for $\rho\in[q,\infty)$,
\begin{equation}\label{eq:truncated-bound}
\mathbb E\Big[\sup_{t}\|u(t,\cdot)\|_{L^q}^\rho\Big]\le\liminf_{R\rightarrow\infty}\mathbb E\Big[\mathbf 1_{\{T\le \tau_R\}}\sup_{t}\|\bar{u}_R(t,\cdot)\|_{L^q}^\rho\Big]\le C(T,\rho,q),
\end{equation}
which also holds for any $\rho\ge1$ and $q\ge1$ in view of the H\"older inequality and the continuity of $u_0$.
Besides, under Assumption \ref{A1}, a standard Picard approximation argument shows that Eq.\ \eqref{CH} admits a unique solution satisfying \eqref{eq:truncated-bound}.

Similar to \cite[Lemma 1.8]{CC01}, we have the following regularity of $G$.
\begin{lem}\label{Greg}
For $\alpha\in(0,1)$, there exists $C=C_\alpha$ such that for $x,y\in\OO$ and $t>s$,
\begin{gather*}
\int_0^t\int_{\OO}|G_{t-r}(x,z)-G_{t-r}(y,z)|^2\ud z\ud r\le C|x-y|^2,\\
\int_0^s\int_{\OO}|G_{t-r}(x,z)-G_{s-r}(x,z)|^2\ud z\ud r+\int_s^t\int_{\OO}|G_{t-r}(x,z)|^2\ud z\ud r\le C|t-s|^{\frac{3}{4}\alpha}.
\end{gather*}
\end{lem}

Based on \eqref{eq:truncated-bound} and Lemma \ref{Greg}, we present the H\"older continuity of the exact solution. Under either Assumption \ref{A1} or Assumption \ref{A2}, there exists some constant $K_0$ such that 
\begin{align}\label{K0}
|f(x)|\le K_0(1+|x|^3).
\end{align}

\begin{lem}\label{Holder-exact}
Let Assumption \ref{A1} or \ref{A2} hold, $u_0\in\mathcal C^2(\mathcal O)$, and $\alpha\in(0,1)$. Then for $p\ge1$, there exists some constant $C=C(\alpha,p,T,K_0)$ such that
\begin{align}\label{eq:Holde-e}
\|u(t,x)-u(s,y)\|_p\le C(|t-s|^{\frac{3\alpha}{8}}+|x-y|),\quad\forall~(t,x),(s,y)\in[0,T]\times\OO.
\end{align}
\end{lem}
\begin{proof}
We first prove \begin{align}\label{eq:bound-e}
\sup_{t,x}\|u(t,x)\|_p\le C(p,T).
\end{align}
To this end, we write $u(t,x)=\sum_{i=1}^3u_i(t,x)$ with 
\begin{gather}\label{eq.u1}
u_1(t,x):=\GG_tu_0(x),\\\label{eq.u2}
u_2(t,x):=\int_0^t\int_{\mathcal O}\Delta G_{t-r}(x,z)f(u(r,z))\ud z\ud r,\\\label{eq.u3}
u_3(t,x):=\int_0^t\int_{\mathcal O}G_{t-r}(x,z)\sigma(u(r,z))W(\ud r,\ud z).
\end{gather}
Let us state a useful property in \cite[Lemma 1.6]{CC01}. 
For any $\rho\in[1,\infty]$, $q\in[\rho,+\infty]$, and $1/\gamma=1/q-1/\rho+1\in[0,1]$,
the linear operator $\mathbb J_{t_0}$ defined by
$$\mathbb J_{t_0}(v)(t,x)=\int_{t_0}^t\int_{\mathcal O}\Delta G_{t-s}(x,y)v(s,y)\ud y\ud s,\quad 0\le t_0<t\le T,~x\in\OO,\quad v\in L^1(t_0,T;L^\rho(\mathcal O))$$
 is a mapping from $L^1(t_0,T;L^\rho(\mathcal O))$ to $L^\infty(t_0,T;L^q(\mathcal O))$ with
\begin{align}\label{AKM}
\|\mathbb J_{t_0}(v)(t,\cdot)\|_{L^q}\le C\int_{t_0}^t(t-s)^{-\frac{3}{4}+\frac{1}{4\gamma}}\|v(s,\cdot)\|_{L^\rho}\ud s.
\end{align}
It follows from \eqref{eq:truncated-bound} and \eqref{K0} that 
\begin{align}\label{f-bound}
\mathbb E[\|f(u(r,\cdot))\|^p_{L^1}]\le C\big(1+\mathbb E[\|u(r,\cdot)\|^{3p}_{L^3}]\big)\le C(p,T,K_0),\quad\forall~r\in[0,T].
\end{align}
Applying \eqref{AKM} with $t_0=0$, $q=\gamma=\infty$ and $\rho=1$ leads to
\begin{align*}
\left\|u_2(t,\cdot)\right\|_{L^\infty}\le C\int_0^t(t-r)^{-\frac{3}{4}}\|f(u(r,\cdot))\|_{L^{1}}\ud r,
\end{align*}
which combined with the Minkowski inequality and \eqref{f-bound} implies
\begin{align}\label{u2p}
\sup_{x}\left\|u_2(t,x)\right\|_{p}\le C\int_0^t(t-r)^{-\frac{3}{4}}\left(\mathbb E[\|f(u(r,\cdot))\|^p_{L^1}]\right)^{\frac{1}{p}}\ud r\le C(p,T,K_0)t^{\frac{1}{4}}.
\end{align}
Since $\sigma$ is bounded, the Burkholder inequality \cite[Theorem B.1]{KD14} and \eqref{Gtxy} yield
\begin{align}\label{u3p}
\left\|u_3(t,x)\right\|^2_{p}\le C(p)\int_0^t\int_{\mathcal O}G^2_{t-r}(x,z)\ud z\ud r\le C(p)t^{\frac{3}{4}},
\end{align}
for $(t,x)\in[0,T]\times\OO$.
In addition, \eqref{Gtxy} implies $|u_1(t,x)|\le C\|u_0\|_{\C(\OO)}$ for $x\in\OO$, which together with \eqref{u2p} and \eqref{u3p} completes the proof of \eqref{eq:bound-e}.

Without loss of generality, assume that $s<t$.
Notice that $u(t,x)-u(t,y)=\GG_tu_0(x)-\GG_tu_0(y)+I_f+I_\sigma$ with
\begin{align*}
I_f:&=\int_0^t\int_{\mathcal O}\left[\Delta G_{t-r}(x,z)-\Delta G_{t-r}(y,z)\right]f(u(r,z))\ud z\ud r,\\
I_\sigma:&=\int_0^t\int_{\mathcal O} \left[G_{t-r}(x,z)- G_{t-r}(y,z)\right]\sigma(u(r,z))W(\ud r,\ud z),
\end{align*}
and $u(t,x)-u(s,x)=\GG_tu_0(x)-\GG_su_0(x)+J^-_f+J_f+J_\sigma+J^-_\sigma$ with
\begin{align*}
J^-_f:&=\int_0^s\int_{\mathcal O}\left[\Delta G_{t-r}(x,z)-\Delta G_{s-r}(x,z)\right]f(u(r,z))\ud z\ud r,\\
J^-_\sigma:&=\int_0^s\int_{\mathcal O} \left[G_{t-r}(x,z)- G_{s-r}(x,z)\right]\sigma(u(r,z))W(\ud r,\ud z),\\
J_f:&=\int_s^t\int_{\mathcal O}\Delta G_{t-r}(x,z)f(u(r,z))\ud z\ud r,~\\
J_\sigma:&=\int_s^t\int_{\mathcal O} G_{t-r}(x,z)\sigma(u(r,z))W(\ud r,\ud z),
\end{align*}
where the explicit dependence of $I_f$, $I_\sigma$, $J^-_f$, $J_f$, $J_\sigma$, $J^-_\sigma$ on $t,s,x,y$ is dropped for simplicity.
Using \cite[Lemma 2.3]{CC01}
and the assumption $u_0\in\mathcal C^{2}(\OO)$, we get
$|\GG_tu_0(x)-\GG_tu_0(y)|+|\GG_tu_0(x)-\GG_su_0(x)|\le C(|t-s|^{\frac{1}{2}}+|x-y|).$
By the Burkholder inequality, the boundedness of $\sigma$, and Lemma \ref{Greg}, we obtain
$\|I_\sigma\|_p^2+\|J_\sigma^-\|_p^2+\|J_\sigma\|_p^2\le C(|t-s|^{\frac{3}{4}\alpha}+|x-y|^2).$
For any $\alpha_1>0$,
\begin{align}\label{ex<x}
e^{-x}\le C_{\alpha_1}x^{-\alpha_1},\quad\forall~ x >0.
\end{align}
By the orthogonality of $\{\phi_j\}_{j=1}^\infty$ in $L^2(\OO)$, \eqref{ex<x}, and $|\phi_j(x)-\phi_j(y)|\le j|x-y|$ for $j\ge1,$ 
\begin{align*}
\int_\OO|\Delta G_{s}(x,z)-\Delta G_{s}(y,z)|^2\ud z= \sum_{j=1}^\infty j^4e^{-2j^4s}|\phi_j(x)-\phi_j(y)|^2\le C |x-y|^2\sum_{j=1}^\infty j^{6-4\rho}s^{-\rho}
\end{align*}
with $\rho>0$. Choosing $\rho\in(\frac{7}{4},2)$ and using the Cauchy--Schwarz inequality,
\begin{align}\label{DeltaGG}\notag
&\quad\int_0^t\int_\OO|\Delta G_{s}(x,z)-\Delta G_{s}(y,z)|\ud z\ud s\le
 \sqrt{\pi}\int_0^t\left(\int_\OO|\Delta G_{s}(x,z)-\Delta G_{s}(y,z)|^2\ud z\right)^{\frac12}\ud s\\
 &\le C|x-y| \Big(\sum_{j=1}^\infty j^{6-4\rho}\Big)^{\frac{1}{2}}\int_0^ts^{-\frac{\rho}{2}}\ud s\le C|x-y|.
\end{align}
 Taking advantage of \eqref{DeltaGG} and \eqref{eq:bound-e}, we obtain
\begin{align*}
\|I_f\|_p&\le C |x-y|\sup_{t,x}\|f(u(t,x))\|_{p}\le C |x-y|(1+\sup_{t,x}\|u(t,x)\|^3_{3p})\le C|x-y|.
\end{align*}
Similarly, the orthogonality of $\{\phi_j\}_{j=1}^\infty$ in $L^2(\OO)$, the Cauchy--Schwarz inequality,  and \eqref{ex<x} yield that for $\rho>\frac{5}{4}$,
\begin{align}\label{eq.Green}\notag
&\quad\int_s^t\int_{\mathcal O}|\Delta G_{t-r}(x,z)|\ud z\ud r\le \sqrt{\pi}\int_s^t\left(\int_{\mathcal O}|\Delta G_{t-r}(x,z)|^2\ud z\right)^{\frac12}\ud r\\
&\le  C\int_s^t\Big(\sum_{j=1}^\infty j^4e^{-2j^4(t-r)}\Big)^{\frac12}\ud r\le C\int_s^t\Big(\sum_{j=1}^\infty j^{4-4\rho}(t-r)^{-\rho}\Big)^{\frac12}\ud r\le C(t-s)^{1-\frac{\rho}{2}},
\end{align}
which together with \eqref{eq:bound-e} implies that for $\alpha\in(0,1)$,
\begin{align*}
\|J_f\|_p\le \int_s^t\int_{\mathcal O}|\Delta G_{t-r}(x,z)| (1+\|u(r,z)\|^3_{3p})\ud z\ud r\le C(t-s)^{\frac{3}{8}\alpha}.
\end{align*}
Observe that for any $\alpha_2\in(0,1]$, 
\begin{align}\label{1-ex}
1-e^{-x}\le C_{\alpha_2}x^{\alpha_2},\quad x\ge 0.
\end{align}
Then the orthogonality of $\{\phi_j\}_{j=1}^\infty$ in $L^2(\OO)$, the Cauchy--Schwarz inequality, \eqref{1-ex} and \eqref{ex<x} imply that 
for $\alpha_3\in(0,\frac{3}{8})$ and $\rho=\frac{13}{8}+\alpha_3<2$ with $4-4\rho+8\alpha_3<-1$, it holds that
\begin{align}\label{eq.Green1}\notag
&\quad\int_0^s\int_\OO|\Delta G_{t-r}(x,z)-\Delta G_{s-r}(x,z)|\ud z\ud r\\\notag
&\le \sqrt{\pi}\int_0^s\left(\int_\OO|\Delta G_{t-r}(x,z)-\Delta G_{s-r}(x,z)|^2\ud z\right)^\frac{1}{2}\ud r\\\notag
&\le C \int_0^s\Big(\sum_{j=1}^\infty j^4e^{-2j^4(s-r)}|1-e^{-j^4(t-s)}|^2\Big)^{\frac{1}{2}}\ud r\\
&\le C\Big(\sum_{j=1}^\infty j^{4-4\rho+8\alpha_3}\Big)^{\frac{1}{2}}\int_0^s (s-r)^{-\frac{\rho}{2}}\ud r(t-s)^{\alpha_3}\le C(t-s)^{\alpha_3}.
\end{align}
Hence, 
 it follows from \eqref{eq:bound-e} that
$\|J_f^-\|_p\le C(t-s)^{\frac{3}{8}\alpha}$ with $\alpha\in(0,1)$.
Finally, combining the above estimates, we obtain \eqref{eq:Holde-e}, which completes the proof.
\end{proof}

%

\begin{lem}\label{Holder-e}
Under the same conditions of Lemma \ref{Holder-exact}, for any $p\ge1$,
\begin{align*}
\mathbb E\Big[\sup_{t,x}|u(t,x)|^p\Big]&\le C(p,T,K_0).
\end{align*}
\end{lem}
\begin{proof}
Based on \eqref{eq:Holde-e} with $\alpha=\frac23$,  we apply \cite[Theorem C.6]{KD14} with $H_1=\frac{1}{4}, H_2=1$,$H=5$, $k=p>6$, $q=1-\frac{6}{p}$ and $\delta=\frac{1}{2}(q+1-H/p)$ to obtain that there exists  $C_1$ such that
\begin{align*}
\mathbb E\bigg[\sup_{(t,x)\neq (s,y)}\frac{|u(t,x)-u(s,y)|^p}{(|t-s|^{\frac{1}{4}}+|x-y|)^{p-6}}\bigg]\le C_1.
\end{align*}
Hence, for $p>6$,
\begin{align}\label{uEp-1}
\mathbb E\Big[\sup_{t,x}|u(t,x)|^p\Big]=\mathbb E\Big[\sup_{t,x}|u(t,x)-u(0,0)|^p\Big]&\le C_1\big(T^{\frac{1}{4}}+\pi\big)^{p-6}.
\end{align}
For $1\le p\le 6$, the desired result follows from \eqref{uEp-1} and the H\"older inequality. The proof is completed.
\end{proof}

\section{Finite difference method}\label{S3}
In this section, we introduce the spatial FDM method for Eq.\ \eqref{CH} and
derive its strong convergence rate.
Given a function $w$ on the mesh $\OO^h:=\{0,h,2h,\ldots,\pi\}$,
we define the difference operator $$\delta_h w_i:=\frac{w_{i-1}-2w_i+w_{i+1}}{h^2},\quad\delta_h^2 w_i=\frac{w_{i-2}-4w_{i-1}+6w_i-4w_{i+1}+w_{i+2}}{h^4},$$
for $i\in \Z_n:=\{1,2,\ldots,n-1\}$,
where $w_i:=w(ih)$. The compatibility conditions  $u_0(0)=u_0(\pi)=0$ and $u_0^{\prime\prime}(0)=u_0^{\prime\prime}(\pi)=0$ are direct results of DBCs and the initial condition. One can approximate $u(t,kh)$ via $\{u^n(t,kh)\}_{n\ge2}$, where $u^n(0,kh)=u_0(kh)$ and
\begin{align}\label{unt}\notag
&\quad\ud u^n(t,kh)+\delta_h^2u^n(t,kh)\ud t\\
&=\delta_h f(u^n(t,kh))\ud t+n\pi^{-1}\sigma(u^n(t,kh))\ud(W(t,(k+1)h)-W(t,kh)),
\end{align}
for $t\in[0,T]$ and $k\in\Z_{n}$,
under the boundary conditions $$u^n(t,0)=u^n(t,\pi)=0,\quad
u^n(t,-h)+u^n(t,h)=u^n(t,(n-1)h)+u^n(t,(n+1)h)=0,$$
for $t\in(0,T]$.
For  $k\in\Z_{n}\cup\{0\}$, we use the polygonal interpolation
$$u^n(t,x):=u^n(t,kh)+(n\pi^{-1}x-k)(u^n(t,(k+1)h)-u^n(t,kh)),\quad\forall~x\in[kh,(k+1)h].$$
To solve \eqref{unt}, we introduce $$U(t)=(U_1(t),\ldots,U_{n-1}(t))^\top,\qquad \beta_t=(\beta^1_t,\ldots,\beta^{n-1}_t)^{\top}$$ with
$U_k(t):=u^n(t,kh)$ and $\beta^k_t:=\sqrt{\frac{n}{\pi}}(W(t,(k+1)h)-W(t,kh))$ for $k\in\mathbb Z_n$, where the explicit dependence of $U(t)$ and $\beta_t$ on $n$ is omitted.
Let 
\begin{small}
\begin{align*}
A_n:=\frac{n^2}{\pi^2}\left[\begin{array}{cccccc}-2 & 1 & 0 & \cdots & \cdots & 0 \\1 & -2 & 1 & \ddots & \ddots & \vdots \\0 & 1 & -2 & \ddots & \ddots & \vdots \\\vdots & 0 & \ddots & \ddots & \ddots & 0 \\\vdots & \ddots & \ddots & 1 & -2 & 1 \\0 & \cdots & \cdots & 0 & 1 & -2\end{array}\right].
\end{align*}
\end{small}
Then \eqref{unt} can be rewritten into an $(n-1)$-dimensional SDE
\begin{align}\label{NCH}
\ud U(t)+A_n^2U(t)\ud t= A_nF_n(U(t))\ud t+\sqrt{n/\pi}\Sigma_n(U(t))\ud \beta_t
\end{align}
with the initial condition $U(0)=(u_0(h),\ldots,u_0((n-1)h))^\top$ and the coefficients $$F_n(U(t))=(f(U_1(t)),\ldots,f(U_{n-1}(t)))^\top,~\Sigma_n(U(t))=\textrm{diag}(\sigma(U_1(t)),\ldots,\sigma(U_{n-1}(t))).$$
Under Assumption \ref{A1}, \eqref{NCH} admits a unique strong solution which satisfies
\begin{align}\label{Unt}\notag
U(t)&=\exp(-A_n^2t)U(0)+\int_0^tA_n\exp(-A_n^2(t-s))F_n(U(s))\ud s\\
&\quad+\sqrt{\frac{n}{\pi}}\int_0^t\exp(-A_n^2(t-s))\Sigma_n(U(s))\ud \beta_s,\quad t\in[0,T].
\end{align}
For $j\in\Z_{n}$, $e_j=(e_j(1),\ldots,e_j(n-1))^\top$ given by
\begin{equation}\label{ejk}
e_j(k)
=\sqrt{\pi/n}\phi_j(kh)=\sqrt{2/n}\sin(jkh),\quad k\in\Z_{n},
\end{equation} 
is an eigenvector of $A_n$  associated with the eigenvalue $\lambda_{j,n}=-j^2c_{j,n}$, where $$c_{j,n}:=\sin^2(\frac{j}{2n}\pi)/(\frac{j}{2n}\pi)^2$$ satisfies $\frac{4}{\pi^2}\le c_{j,n}\le 1.$ The vectors $\{e_i\}_{i=1}^{n-1}$  form an orthonormal basis of $\R^{n-1}$. In particular,
\begin{align}\label{orth}
\langle e_i,e_j\rangle=\frac{\pi}{n}\sum_{k=1}^{n-1}\phi_i(kh)\phi_j(kh)=\int_0^\pi\phi_i(\kappa_n(y))\phi_j(\kappa_n(y))\ud y=\delta_{ij},
\end{align} where $\kappa_n(y)=h\lfloor y/h\rfloor$  with $\lfloor\cdot\rfloor$ being the floor function (see e.g., \cite{GI98}).
It is verified that
$1-\frac{\sin a}{a}\le \frac{1}{6}a^2$ for all $a\in[0,\frac{\pi}{2})$, which indicates that for $j\in\Z_n$,
\begin{align}\label{cjn}
0\le 1-c_{j,n}=\Big(1+\sin(\frac{j}{2n}\pi)/(\frac{j}{2n}\pi)\Big)\Big(1-\sin(\frac{j}{2n}\pi)/(\frac{j}{2n}\pi)\Big)\le \frac{\pi^2j^2}{12n^2}.
\end{align}
Introduce the discrete kernel
$$G^n_t(x,y)=\sum_{j=1}^{n-1}\exp(-\lambda_{j,n}^2t)\phi_{j,n}(x)\phi_j(\kappa_n(y)),$$
where $\phi_{j,n}(x)=\phi_{j}(kh)+(n\pi^{-1}x-k)(\phi_{j}((k+1)h)-\phi_{j}(kh))$ for $x\in[kh,(k+1)h]$, $k\in\Z_n\cup\{0\}$.
Define the discrete Dirichlet Laplacian $\Delta_n$ by $\Delta_n w(y)=0$ for $y\in[0,h)$, 
\begin{align}\label{Deltaw}
\Delta_n w(y)=\frac{n^2}{\pi^2}\left(w\big(\kappa_n(y)+\frac{\pi}{n}\big)-2w\big(\kappa_n(y)\big)+w\big(\kappa_n(y)-\frac{\pi}{n}\big)\right),\quad y\in[h,\pi),
\end{align}
where $w:\OO\rightarrow \R$ with $w(0)=w(\pi)=0$.
Since $\Delta_n \phi_j(\kappa_n(y))=\lambda_{j,n}\phi_j(\kappa_n(y))$,
$$\Delta_n G^n_{t}(x,y)=\sum_{j=1}^{n-1}\lambda_{j,n}\exp(-\lambda_{j,n}^2t)\phi_{j,n}(x)\phi_j(\kappa_n(y)).$$
Similar to \cite[Section 2]{GI98}, based on \eqref{Unt}, the diagonalization of the matrix $A_n$, \eqref{ejk} and $u^n(t,kh)=\sum_{j=1}^{n-1}\langle U(t),e_j\rangle e_j(k)$, one  has 
\begin{align}\label{unR}\notag
u^n(t,x)=&\int_{\mathcal O}G^n_t(x,y)u_0(\kappa_n(y))\ud y+\int_0^t\int_{\mathcal O}\Delta_nG^n_{t-s}(x,y)f(u^n(s,\kappa_n(y)))\ud y\ud s\\
&+\int_0^t\int_{\mathcal O}G^n_{t-s}(x,y)\sigma(u^n(s,\kappa_n(y)))W(\ud s,\ud y),\quad(t,x)\in[0,T]\times\OO.
\end{align}

The follow lemma characterizes the error between $G$ and $G^n$. \begin{lem}\label{lem:Gn-G}
 There exists some constant $C=C(T)$ such that for any $x\in\OO$ and $t\in(0,T]$,
\begin{align}\label{DeltaGnG}
\int_0^t\int_{\mathcal O}|\Delta_nG_{s}^n(x,y)-\Delta G_{s}(x,y)|\ud y\ud s&\le Cn^{-1},
\\\label{Gn-G}
\int_0^t\int_{\mathcal O}|G^{n}_s(x,y)-G_s(x,y)|^2\ud y\ud s&\le Cn^{-2}.
\end{align}
\end{lem}
\begin{proof}
For $s\in[0,T]$ and $x,y\in\OO$, denote $M_1^{s,x,y}=\sum_{j=1}^{n-1}\lambda_{j,n}e^{-\lambda_{j,n}^2s}\phi_{j,n}(x)\left(\phi_j(\kappa_n(y))-\phi_j(y)\right)$, $M_2^{s,x,y}=\sum_{j=1}^{n-1}(\lambda_{j,n}e^{-\lambda_{j,n}^2s}\phi_{j,n}(x)-\lambda_je^{-\lambda_{j}^2s}\phi_{j}(x))\phi_j(y)$ and $M_3^{s,x,y}=\sum_{j=n}^{\infty}-\lambda_je^{-\lambda_{j}^2s}\phi_{j}(x)\phi_j(y)$.
Since $\{\phi_j\}_{j\ge1}$ is an orthonormal basis of $L^2(\OO)$, it holds that 
\begin{align}\label{eq.Delta-Delta}\notag
&\quad\int_\OO|\Delta_nG_{s}^n(x,y)-\Delta G_{s}(x,y)|^2\ud y=
\int_\OO|\sum_{i=1}^3M_i^{s,x,y}|^2\ud y\\
&\le 3\int_\OO|M_1^{s,x,y}|^2\ud y+3\sum_{j=1}^{n-1}\left|\lambda_{j,n}e^{-\lambda_{j,n}^2s}\phi_{j,n}(x)-\lambda_je^{-\lambda_{j}^2s}\phi_{j}(x)\right|^2+3\sum_{j=n}^{\infty}j^4e^{-2j^4 s}.
\end{align}
It follows from the boundedness of $\{\phi_j\}_{j\ge1}$ and \eqref{ex<x} that for $\alpha_1\in(\frac54,2)$,
\begin{align}\label{eq.Delta-Delta1}
\int_0^t\Big(\sum_{j=n}^{\infty}j^4e^{-2j^4 s}\Big)^{\frac12}\ud s\le C\int_0^t\Big(\sum_{j=n}^{\infty}j^{4-4\alpha_1}s^{-\alpha_1}\Big)^{\frac12}\ud s\le C(\alpha_1)n^{\frac{5-4\alpha_1}{2}}.
\end{align}
By \eqref{cjn}, we have 
$|\lambda_j-\lambda_{j,n}|\le Cj^4/n^2,$
and thus,
$\lambda_j^2-\lambda_{j,n}^2=|\lambda_j-\lambda_{j,n}||\lambda_j+\lambda_{j,n}|\le Cj^6/n^2$, which along with \eqref{1-ex} yields $|e^{-\lambda_{j,n}^2s}-e^{-\lambda_{j}^2s}|=e^{-\lambda_{j,n}^2s}|1-e^{-(\lambda_{j}^2-\lambda_{j,n}^2)s}|\le e^{-\lambda_{j,n}^2s}\frac{j^6}{n^2}s$. Besides, it can be verified that $|\phi_{j,n}(x)-\phi_j(x)|\le Cj/n.$ Therefore, for $\rho,\rho_1>0$
\begin{align*}
&\quad\sum_{j=1}^{n-1}\left|\lambda_{j,n}e^{-\lambda_{j,n}^2s}\phi_{j,n}(x)-\lambda_je^{-\lambda_{j}^2s}\phi_{j}(x)\right|^2\\
&\le\sum_{j=1}^{n-1}|\lambda_{j,n}-\lambda_j|^2e^{-2\lambda_{j,n}^2s}+\sum_{j=1}^{n-1}\lambda_j^2|e^{-\lambda_{j,n}^2s}-e^{-\lambda_{j}^2s}|^2+\sum_{j=1}^{n-1}\lambda_j^2e^{-2\lambda_{j}^2s}|\phi_{j,n}(x)-\phi_j(x)|^2\\
&\le Cn^{-4}\sum_{j=1}^{n-1}j^{8-4\rho}s^{-\rho}+Cn^{-4}\sum_{j=1}^{n-1}{j^{16-4\rho_1}}s^{2-\rho_1}+Cn^{-2}\sum_{j=1}^{n-1}j^{6-4\rho}s^{-\rho}.
\end{align*}
Choosing $\rho=\rho_1-2\in(0,2)$, we obtain
\begin{align}\label{eq.Delta-Delta2}
\int_0^t\Big(\sum_{j=1}^{n-1}\left|\lambda_{j,n}e^{-\lambda_{j,n}^2s}\phi_{j,n}(x)-\lambda_je^{-\lambda_{j}^2s}\phi_{j}(x)\right|^2\Big)^{\frac12}\ud s\le C\int_0^t(n^{5-4\rho}s^{-\rho})^{\frac12}\ud s\le Cn^{\frac{5-4\rho}{2}}.
\end{align}
We recall the following inequality in \cite[Lemma 3.2]{GI98}: 
\begin{align}\label{eq.key}
\int_\OO|w(y)-w(\kappa_n(y))|^2\ud y\le Cn^{-2}\int_\OO|\frac{\ud}{\ud y}w(y)|^2\ud y,\quad\text{for}~w\in C^1(\OO).
\end{align}
Introducing $B_n(s,x,y)=\sum_{j=1}^{n-1}\lambda_{j,n}e^{-\lambda_{j,n}^2s}\phi_{j,n}(x)\phi_j(y)$ and making use of \eqref{eq.key}, we obtain 
\begin{align*}
&\quad\int_\OO|M_1^{s,x,y}|^2\ud y= \int_\OO|B_n(s,x,y)-B_n(s,x,\kappa_n(y))|^2\ud y
\le Cn^{-2}\sum_{j=1}^{n-1}j^6e^{-2\lambda_{j,n}^2s}\phi_{j,n}(x)^2.
\end{align*}
Hence, it follows from \eqref{ex<x} that for $\rho\in(\frac74,2)$,
\begin{align}\label{eq.Delta-Delta3}
\int_0^t\Big(\int_\OO|M_1^{s,x,y}|^2\ud y\Big)^{\frac12}\ud s\le C\int_0^t\Big(\frac{1}{n^2}\sum_{j=1}^{n-1}j^{6-4\rho}s^{-\rho}\Big)^{\frac12}\ud s\le Cn^{-1}.
\end{align}
Combining \eqref{eq.Delta-Delta}-\eqref{eq.Delta-Delta2} with \eqref{eq.Delta-Delta3} allows us to deduce
\begin{align*}
\int_0^t\int_{\mathcal O}|\Delta_nG_{s}^n(x,y)-\Delta G_{s}(x,y)|\ud y\ud s\le \sqrt{\pi}\int_0^t\left(\int_{\mathcal O}|\Delta_nG_{s}^n(x,y)-\Delta G_{s}(x,y)|^2\ud y\right)^{\frac12}\ud s\le Cn^{-1},
\end{align*}
which proves \eqref{DeltaGnG}.
It remains to prove \eqref{Gn-G}.
Set $H_n(t,x,y):=\sum_{j=1}^{n-1}\exp(-\lambda_j^2t)\phi_{j}(x)\phi_j(y).$
Then 
$\int_{\mathcal O}|G^{n}_s(x,y)-G_s(x,y)|^2\ud y\le 4\sum_{k=1}^4J_k(s,x),$
where
\begin{gather*}
J_1(s,x):=\sum_{j=n}^{\infty}\exp(-2\lambda_j^2s),\qquad
J_2(s,x):=\int_\OO|H_n(s,x,y)-H_n(s,x,\kappa_n(y))|^2\ud y,\\
J_3(s,x):=\sum_{j=1}^{n-1}|\exp(-\lambda_j^2s)-\exp(-\lambda_{j,n}^2s)|^2,\qquad J_4(s,x):=\sum_{j=1}^{n-1}e^{-2\lambda_{j,n}^2s}|\phi_{j,n}(x)-\phi_j(x)|^2.
\end{gather*}
For the first term, we have $\int_0^t J_1(s,x)\ud s\le C\sum_{j=n}^\infty j^{-4}\le Cn^{-3}$.
For $2<\alpha<3$,
$
J_3(s,x)
\le \sum_{j=1}^{n-1}e^{-2\lambda_{j,n}^2s}(1-e^{-(\lambda_j^2-\lambda_{j,n}^2)s})^2\le C\sum_{j=1}^{n-1}\frac{j^{12}}{n^4}j^{-4\alpha}s^{2-\alpha}\le Cn^{9-4\alpha}s^{2-\alpha},
$
which implies that $\int_0^t J_3(s,x)\ud s\le C_\epsilon n^{-3+\epsilon}$ with arbitrarily small $\epsilon>0$. Since $|\phi_{j,n}(x)-\phi_j(x)|\le Cj/n$, it holds  that
$J_4(s,x)\le C\sum_{j=1}^{n-1}j^{-4\alpha}s^{-\alpha}\frac{j^2}{n^2}\le Cn^{-2}s^{-\alpha}$ for $\frac{3}{4}<\alpha<1$,
and thus $\int_0^t J_4(s,x)\ud s\le Cn^{-2}.$
Using \eqref{eq.key},
we arrive at
\begin{align*}
\int_0^tJ_2(s,x)\ud s
\le Cn^{-2}\int_0^t\int_\OO|\frac{\ud}{\ud y}H_n(s,x,y)|^2\ud y\ud s\le Cn^{-2}\sum_{j=1}^{n-1}\int_0^tj^2\exp(-2j^4s)\ud s\le Cn^{-2}.
\end{align*}
Combining the above estimates completes the proof of \eqref{Gn-G}.
\end{proof}
For $n\ge2$, denote by $\mathbb U(t):=(u(t,h),\ldots,u(t,(n-1)h))^\top$ the exact solution of Eq.\ \eqref{CH} on spatial grid points, where the explicit dependence of $\mathbb U(t)$ on $n$ is omitted. We introduce the following auxiliary process $\{\tilde U(t),t\in[0,T]\}$  by
\begin{align*}
\ud \tilde U(t)+A_n^2\tilde U(t)\ud t= A_nF_n(\mathbb U(t))\ud t+\sqrt{n/\pi}\Sigma_n(\mathbb U(t))\ud \beta_t,\quad t\in(0,T]
\end{align*}
with initial value $\tilde U(0)=U(0)$. Let $\tilde u^n=\{\tilde u^n(t,x),(t,x)\in[0,T]\times\OO\}$  satisfy 
\begin{align}\label{eq.tildeu}\notag
\tilde u^n(t,x)=&\int_{\mathcal O}G^n_t(x,y)u_0(\kappa_n(y))\ud y+\int_0^t\int_{\mathcal O}\Delta_nG^n_{t-s}(x,y)f(u(s,\kappa_n(y)))\ud y\ud s\\
&+\int_0^t\int_{\mathcal O}G^n_{t-s}(x,y)\sigma(u(s,\kappa_n(y)))W(\ud s,\ud y).
\end{align}
Then $\tilde U_k(t)=\tilde u^n(t,kh)$ for $k\in\Z_{n}$ and $t\in[0,T]$. In order to estimate $\|u^n(t,x)-u(t,x)\|_p$, it suffices to estimate $\|\tilde u^n(t,x)-u(t,x)\|_p$ and $\|\tilde u^n(t,x)-u^n(t,x)\|_p$, where the  first term is tackled as follows.

\begin{lem}\label{utilde-uh-1}
Suppose that Assumption \ref{A1} or \ref{A2} holds and  $u_0\in\C^3(\OO)$.
 Then for any $p\ge1$, there exists some constant $C=C(p,T,K_0)$ such that for any $(t,x)\in[0,T]\times\OO$,
\begin{align*}
\|\tilde u^n(t,x)-u(t,x)\|_p\le Cn^{-1}.
\end{align*}
\end{lem}
\begin{proof}
Recall that $u=u_1+ u_2+u_3$, where $u_i$, $i=1,2,3$, are defined in \eqref{eq.u1}-\eqref{eq.u3}, respectively. Similarly, for $(t,x)\in[0,T]\times\OO$,
we introduce $\tilde u^n_1(t,x):=\int_{\mathcal O}G^n_t(x,y)u_0(\kappa_n(y))\ud y$,
\begin{gather*}
\tilde u^n_2(t,x):=\int_0^t\int_{\mathcal O}\Delta_nG^n_{t-s}(x,y)f(u(s,\kappa_n(y)))\ud y\ud s,\\
\tilde u^n_3(t,x):=\int_0^t\int_{\mathcal O}G^n_{t-s}(x,y)\sigma(u(s,\kappa_n(y)))W(\ud s,\ud y),
\end{gather*}
and divide the proof into three parts.

\textbf{Part 1:} Following the proof of \cite[Lemma 2.3]{CC01}, we use the PDE satisfied by $G$ to write
$u_1(t,x)=u_0(x)-\int_0^t\int_\OO\Delta G_r(x,z)u_0^{\prime\prime}(z)\ud z\ud r.$ 
As a numerical counterpart, 
\begin{align}\label{eq.u1tilde}\notag
&\quad\tilde u^n_1(t,x)-\tilde u^n(0,x)=\int_{\mathcal O}\int_0^t\frac{\partial}{\partial r}G^n_r(x,z)u_0(\kappa_n(z))\ud z\ud r\\&=-\int_0^t\int_\OO \Delta_n^2 G^n_r(x,z) u_0(\kappa_n(z))\ud z\ud r=
-\int_0^t\int_\OO \Delta_n G^n_r(x,z)\Delta_n u_0(z)\ud z\ud r,
\end{align}
where in the last step we have used the fact that $$\int_\OO\Delta_n v(z)w(\kappa_n(z))\ud z=\int_\OO v(\kappa_n(z))\Delta_n w(z)\ud z,$$
for $v,w:\OO\rightarrow \R$ with $v=w=0$ on $\partial \OO$.
Here, $\tilde u^n(0,kh)=u_0(kh)$ for $k\in\Z_n\cup\{0,n\}$, and 
$\tilde u^n(0,x)=u_0(kh)+(n\pi^{-1}x-k)(u_0((k+1)h)-u_0(kh))$ for $x\in[kh,(k+1)h]$, $k\in\Z_n\cup\{0\}$. In particular, when $u_0\in\mathcal C^1(\OO)$, it holds that 
\begin{align}\label{un0Holder}
|\tilde u^n(0,x)-\tilde u^n(0,y)|\le C|x-y|,\quad x,y\in \OO.
\end{align}
By $u_0\in \C^3(\OO)$ and \eqref{Deltaw}, there exist $\theta_1,\theta_2\in(0,1)$ such that for $z\in[h,\pi)$,
$$ |u^{\prime\prime}_0(z)-\Delta_nu_0(z)|=|u^{\prime\prime}_0(z)-\frac{1}{2}u_0^{\prime\prime}(\kappa_n(z)+\theta_1\frac{\pi}{n})-\frac{1}{2}u_0^{\prime\prime}(\kappa_n(z)-\theta_2\frac{\pi}{n})|\le Cn^{-1},$$
and for $z\in[0,h)$, $|u^{\prime\prime}_0(z)-\Delta_nu_0(z)|=|u^{\prime\prime}_0(z)|=|u^{\prime\prime}_0(z)-u^{\prime\prime}_0(0)|\le Cn^{-1}.$
Therefore, using  \eqref{DGtxy} and \eqref{DeltaGnG}, a direct calculation gives 
\begin{align}\label{deltau1}\notag
|\tilde u^n_1(t,x)-u_1(t,x)|&\le \frac{C}{n}+\int_0^t\int_\OO|\Delta_n G^n_r(x,z)-\Delta G_r(x,z)|\ud z\ud r\\
&\quad+\int_0^t\int_\OO|\Delta G_r(x,z)||u^{\prime\prime}_0(z)-\Delta_nu_0(z)|\ud z\ud r\le Cn^{-1}.
\end{align}

\textbf{Part 2:} The error $\tilde u^n_3(t,x)-u_3(t,x)$ is divided into
\begin{align*}
\tilde u^n_3(t,x)-u_3(t,x)=&\int_0^t\int_{\mathcal O}[G_{t-s}^n(x,y)-G_{t-s}(x,y)]\sigma(u(s,\kappa_n(y)))W(\ud s,\ud y)\\
&+\int_0^t\int_{\mathcal O}G_{t-s}(x,y)[\sigma(u(s,\kappa_n(y)))-\sigma(u(s,y))]W(\ud s,\ud y).
\end{align*}
The Burkholder inequality, the boundedness and Lipschitz continuity of $\sigma$, \eqref{Gn-G}, \eqref{Gtxy}, and Lemma \ref{Holder-exact} imply 
\begin{align}\label{deltau3}\notag
&\quad\|\tilde u^n_3(t,x)-u_3(t,x)\|_p^2
\\\notag
&\le\int_0^t\int_{\mathcal O}|G_{t-s}^n(x,y)-G_{t-s}(x,y)|^2\ud y\ud s
+C\int_0^t\int_{\mathcal O}G^2_{t-s}(x,y)\|u(s,\kappa_n(y))-u(s,y)\|_p^2\ud y\ud s\\
&\le Cn^{-2}+C\int_0^t(t-s)^{-\frac{1}{2}}\sup_{y\in\mathcal O}\|u(s,\kappa_n(y))-u(s,y)\|_p^2\ud s\le Cn^{-2}.
\end{align}

\textbf{Part 3:} Notice that
$\|\tilde u^n_2(t,x)-u_2(t,x)\|_p\le I_1+I_2$, where
\begin{align*}
&I_1:=\int_0^t\int_{\mathcal O}|\Delta_nG_{t-s}^n(x,y)-\Delta G_{t-s}(x,y)|\|f(u(s,\kappa_n(y)))\|_p\ud y\ud s,\\
&I_2:=\int_0^t\int_{\mathcal O}|\Delta G_{t-s}(x,y)|\left\|f(u(s,\kappa_n(y)))-f(u(s,y))\right\|_p\ud y\ud s.
\end{align*}
It follows from \eqref{DeltaGnG}, \eqref{K0} and Lemma \ref{Holder-e} that 
\begin{align*}
I_1\le C\int_0^t\int_{\mathcal O}|\Delta_nG_{t-s}^n(x,y)-\Delta G_{t-s}(x,y)|\ud y\ud s\Big(1+\sup_{t,x}\|u(t,x)\|_{3p}^3\Big)\le Cn^{-1}.
\end{align*}
Under Assumption \ref{A1} or \ref{A2}, $|f(a_1)-f(a_2)|\le c_0(1+a_1^2+a_2^2)|a_2-a_1|$. Hence, the H\"older inequality together with Lemmas \ref{Holder-exact} and \ref{Holder-e} yields that for any $(s,x)\in[0,T]\times\OO$,
\begin{align*}
&\quad\|f(u(s,\kappa_n(x)))-f(u(s,x))\|_{p}\\
&\le C\|u(s,\kappa_n(x))-u(s,x)\|_{3p}\big(1+\|u(s,\kappa_n(x))\|^2_{3p}+\|u(s,x)\|^2_{3p}\big)
\le Cn^{-1}.
\end{align*}
Therefore, $I_2\le Cn^{-1}\int_0^t\int_{\mathcal O}|\Delta G_{t-s}(x,y)|\ud y\ud s\le Cn^{-1}$, in view of \eqref{Gtxy}.
%
%
%
%
%
In conclusion,
$\|\tilde u^n_2(t,x)-u_2(t,x)\|_p\le C n^{-1},$
which along with \eqref{deltau1} and \eqref{deltau3} finishes the proof.
\end{proof}

By Lemma \ref{utilde-uh-1}, we present the strong convergence rate of the spatial FDM for Eq.\ \eqref{CH}. We would like to mention that Theorem \ref{eq.strong-trun} also holds for stochastic Cahn--Hilliard equations with NBCs. 

\begin{tho}\label{eq.strong-trun}
Suppose that Assumption \ref{A1} holds and $u_0\in\C^3(\OO)$. Then for every $p\ge1$, there exists some constant $C=C(p,T,K)$ such that for any $(t,x)\in[0,T]
\times\OO$,
$$\|u^n(t,x)-u(t,x)\|_{p}\le Cn^{-1}.$$
\end{tho}
\begin{proof}
Denote $e^n(t,x):=u^n(t,x)-u(t,x)$.
In view of \eqref{unR} and \eqref{eq.tildeu},
\begin{align*}
u^n(t,x)-\tilde u^n(t,x)&=\int_0^t\int_{\mathcal O}\Delta_nG^n_{t-s}(x,y)\big[f(u^n(s,\kappa_n(y)))-f(u(s,\kappa_n(y)))\big]\ud y\ud s\\
&\quad+\int_0^t\int_{\mathcal O}G^n_{t-s}(x,y)\big[\sigma(u^n(s,\kappa_n(y)))-\sigma(u(s,\kappa_n(y)))\big]W(\ud s,\ud y).
\end{align*}
By the expressions of $G_t^n$ and $\Delta_nG_t^n$ and  \eqref{ex<x}, we arrive at that for $0<\epsilon\ll 1$,
\begin{align}\label{GnGn}
|G^n_{t}(x,y)|\le C_\epsilon t^{-\frac{1}{4}-\epsilon},\qquad|\Delta_nG^n_{t}(x,y)|\le C_\epsilon t^{-\frac{3}{4}-\epsilon},
\quad\forall~t\in(0,T],~x,y\in\mathcal O.
\end{align}
Hence,  the Cauchy--Schwarz inequality with respect to the measure $|\Delta_n G^n_{t-s}(x,y)|\ud y\ud s$, Lemma \ref{utilde-uh-1}, Assumption \ref{A1}, and the Minkowski and Burkholder inequalities  yield that for $0<\epsilon\ll 1$,
\begin{align*}
\|e^n(t,x)\|_p^2&\le C n^{-2}+C\int_0^t\int_{\mathcal O}|\Delta_n G^n_{t-s}(x,y)|\|u^n(s,\kappa_n(y))-u(s,\kappa_n(y))\|_p^2\ud y\ud s\\
&\quad +C\int_0^t\int_{\mathcal O}|G^n_{t-s}(x,y)|^2\|u^n(s,\kappa_n(y))-u(s,\kappa_n(y))\|_p^2\ud y\ud s\\
&\le C  n^{-2}+C_\epsilon \int_0^t\int_{\mathcal O}(t-s)^{-\frac{3}{4}-\epsilon}\|u^n(s,\kappa_n(y))-u(s,\kappa_n(y))\|_p^2\ud y\ud s,
\end{align*}
in which the second step used \eqref{GnGn}.
Taking the supremum over $x$ produces
\begin{align}\label{entx}
\sup_x\|e^n(t,x)\|^2_p&\le C n^{-2}+C_\epsilon \int_0^t(t-s)^{-\frac{3}{4}-\epsilon}\sup_x\|e^n(s,x)\|_p^2\ud s,
\end{align}
which along with the Gronwall lemma with weak singularities (see e.g., \cite[Lemma 3.4]{GI98}) completes the proof.
\end{proof}
%
\section{Convergence of density}\label{S4}
For real-valued random variables $X,Y$, we write $\mathrm{d}_{\mathrm{TV}}(X,Y)$ to indicate the total variation distance between $X$ and $Y$, i.e.,
$$\mathrm{d}_{\mathrm{TV}}(X,Y)=2\sup_{A\in\mathscr B(\R)}\{|\mathbb P(X\in A)-\mathbb P(Y\in A)|\}=\sup_{\phi\in\Phi}|\E[\phi(X)]-\E[\phi(Y)]|,$$
where $\Phi$ is the set of continuous functions $\phi:\R\rightarrow\R$ which are bounded by $1$, and $\mathscr B(\R)$ is the Borel $\sigma$-algebra of $\R$. Furthermore, if $\{X_n\}_{n\ge1}$ and $X_\infty$ have the densities $p_{X_n}$ and $p_{X_\infty}$ respectively, then 
\begin{align}\label{pX-PX}
\mathrm{d}_{\mathrm{TV}}(X_n,X_\infty)
=\|p_{X_n}-p_{X_\infty}\|_{L^1(\R)}.
\end{align}
In this section, we show that for $k\in\Z_n$,
the spatial semi-discrete numerical solution $u^n(T,kh)$ admits a density, which converges in $L^1(\R)$ to the density of the exact solution $u(T,kh)$.

\subsection{Malliavin calculus}
We start with introducing some notations in the context of the Malliavin calculus with respect to the space-time white noise (see e.g., \cite{DN06}). 
The isonormal Gaussian family $\{W(h),h\in\cH\}$ corresponding to $\cH:=L^2([0,T]\times\OO)$ is  given by the Wiener integral
$W(h)=\int_{0}^T\int_\OO h(s,y)W(\ud s,\ud y).$ 
Denote by
$\mathcal{S}$ the class of smooth real-valued random variables of the form 
\begin{equation}\label{smoothfunctional}
X= \varphi(W(h_1),\ldots,W(h_n)),
\end{equation}
where $\varphi\in\C_p^\infty(\mathbb{R}^n),$ $h_i\in \cH,\, i=1,\ldots,n, \,n\ge 1.$
Here $\C_p^\infty(\mathbb{R}^n)$ is the space of all $\R$-valued smooth functions on $\R^n$ whose partial derivatives have at most polynomial growth.
The Malliavin derivative of $X\in\mathcal S$ of the form \eqref{smoothfunctional} is an $\cH$-valued random variable given by 
$DX=\sum_{i=1}^n\partial_i\varphi(W(h_1),\ldots,W(h_n))h_i,$
which is also a random field $DX=\{D_{\theta,\xi}X, (\theta,\xi)\in[0,T]\times\OO\}$ with
$D_{\theta,\xi}X\!=\sum_{i=1}^n\partial_i\varphi(W(h_1),\ldots,W(h_n))h_i(\theta,\xi)$
for almost everywhere $(\theta,\xi,\omega)\in[0,T]\times\OO\times\Omega$.
For any $p\ge 1$, we denote the domain of $D$ in $L^p(\Omega;\R)$ by $\mathbb{D}^{1,p}$, meaning that $\mathbb{D}^{1,p}$ is the closure of $\mathcal{S}$ with respect to the norm
$$\|X\|_{1,p}=\left(\mathbb{E}\left[|X|^p+\|DX\|_\cH^p\right]\right)^{\frac{1}{p}}.$$
We define the iteration of the operator $D$ in such a way that for $X\in\mathcal S$, the iterated derivative $D^k X$ is an $\cH^{\bigotimes k}$-valued
random variable. More precisely,  for $k\in\mathbb N_+$, $D^k X=\{D_{r_1,\theta_1}\cdots D_{r_k,\theta_k}X,(r_i,\theta_i)\in[0,T]\times\OO\}$  is a measurable function on the product space $([0,T]\times\OO)^k \times \Omega$. 
Then for  $p\ge1$, $k\in\mathbb N$, denote by $\mathbb{D}^{k,p}$ the completion of $\mathcal{S}$ with respect to the norm
$\|X\|_{k,p}=\big(\mathbb{E}\big[|X|^p+\sum_{j=1}^{k}\|D^jX\|_{\cH^{\bigotimes j}}^p\big]\big)^{\frac{1}{p}}.$
Define $\mathbb{D}^{k,\infty}:=\bigcap_{p\ge 1} \mathbb{D}^{k,p}$ and $\mathbb{D}^{\infty}:=\bigcap_{k\ge 1} \mathbb{D}^{k,\infty}$ to be  topological projective limits. 

We close this part by 
the following proposition, which allows us to obtain the convergence of density of a sequence of random variables from the convergence in $\mathbb D^{1,2}$.
\begin{prop}\cite[Theorem 4.2]{NP13}\label{TVconvergence}
Let $\{X_n\}_{n\ge1}$ be a sequence in $\mathbb D^{1,2}$ such that each $X_n$ admits a density. Let $X_\infty\in\mathbb D^{2,4}$ and let $0<\alpha\le 2$ be such that $\E[\|DX_\infty\|_{\cH}^{-\alpha}]<\infty$. If $X_n\rightarrow X_\infty$ in $\mathbb D^{1,2}$, then there exists a constant $c>0$ depending only on $X_\infty$ such that for any $n\ge1$, 
$$\mathrm{d}_{\mathrm{TV}}(X_n,X_\infty)\le c\|X_n-X_\infty\|_{1,2}^{\frac{\alpha}{\alpha+2}}.$$
\end{prop}

\subsection{Convergence in $\mathbb D^{1,2}$}
In this part, we extend the strong convergence of the spatial FDM  to the convergence in $\mathbb D^{1,2}$. 
It is shown in \cite[Proposition 3.1]{CH20} or \cite[Lemma 3.2]{CC01}  that if $f(x)=(x^3-x)K_R(x)$, then
for any $(t,x)\in[0,T]\times\OO$, $u(t,x)\in\D^{1,2}$ and satisfies
\begin{align}\label{Drzchi}\notag
D_{r,z}u(t,x)&=G_{t-r}(x,z)\sigma(u(r,z))+\int_r^t\int_{\mathcal O}\Delta G_{t-s}(x,y)f^\prime(u(s,y))D_{r,z}u(s,y)\ud y\ud s\\
&\quad+\int_r^t\int_{\mathcal O}G_{t-s}(x,y)\sigma^\prime(u(s,y))D_{r,z}u(s,y)W(\ud s,\ud y),
\end{align}
if $r\le t$, and $D_{r,z}u(t,x)=0$, if $r>t$. Their proofs rely on the global Lipschitz continuity of $f_R$, and thus \eqref{Drzchi} holds naturally whenever $f$ satisfies Assumption \ref{A1}. Further, we impose Assumption \ref{A4} to study the regularity of the exact solution in the Malliavin Sobolev space.

\begin{Asp}\label{A4}
For some integer $k\ge1$, $f$ and $\sigma$ have bounded derivatives up to order $k$.
\end{Asp}
\begin{lem}\label{lem:chimp0}
Under Assumptions \ref{A1} and \ref{A4},  $u(t,x)\in \D^{k,\infty}$ for any $(t,x)\in[0,T]\times\OO$. Moreover, for any $p\ge1$, there exists $C=C(k,p,T)$ such that
$$\sup_{t,x}\|u(t,x)\|_{k,p}\le C.$$
\end{lem}
\begin{proof}
Define the Picard approximation by $w^0(t,x)=u_0(x)$, $(t,x)\in[0,T]\times\OO$, and for $i\in\mathbb N$,
\begin{align*}
w^{i+1}(t,x)&=\GG_tu_0(x)+\int_0^t\int_\OO \Delta G_{t-s}(x,y)f(w^{i}(s,y))\ud y\ud s\\
&\quad+\int_0^t\int_\OO G_{t-s}(x,y)\sigma(w^i(s,y))W(\ud s,\ud y),\quad (t,x)\in[0,T]\times\OO.
\end{align*}
Fix $(t,x)\in[0,T]\times\OO$.
In view of \cite[Lemma 1.5.3]{DN06},
the proof of $u(t,x)\in\D^{k,\infty}$ boils down to proving that 

(i) $\{w^i(t,x)\}_{i\ge1}$ converges to $u(t,x)$ in $L^p(\Omega;\R)$ for every $p\ge1$.

(ii) for any  $p\ge1$,
$\sup_{i\ge0}\|w^i(t,x)\|_{k,p}<\infty.$

Property (i) and property (ii) with $k=1$ and $p=2$
can be obtained in the same way as in \cite[Lemma 3.2]{CC01} (the sequence $\{w^i(t,x)\}_{i\ge1}$ corresponds to $\{u_{n,k}(t,x)\}_{k\ge1}$ in \cite{CC01}). 
The proof of property (ii) with general $k,p\ge1$ is omitted since it is standard and similar to those for other kinds of SPDEs with Lipschitz continuous coefficients; see \cite[Proposition 4.3]{BP98} for the case of stochastic heat equations, \cite[Theorem 1]{QS04} for the case of stochastic wave equations.
\end{proof}

Similar to properties (i) and (ii), the standard Picard approximation also shows that 
for any $(t,x)\in[0,T]\times\OO$, $u^n(t,x)\in\D^{1,2}$. 

\begin{prop}\label{eq.strong-trun-D}
Suppose that  $u_0\in\C^3(\OO)$  and Assumptions \ref{A1} and \ref{A4} hold for $k=2$. Then there exists some constant $C$ such that for any $(t,x)\in[0,T]\times\OO$,
$$\E\left[\|Du^n(t,x)-Du(t,x)\|_{\cH}^2\right]\le Cn^{-2}.$$
\end{prop}
\begin{proof}
By the chain rule and \eqref{unR}, we obtain
\begin{align}\notag\label{Drzchintx}
D_{r,z}u^n(t,x)&=G^n_{t-r}(x,z)\sigma(u^n(r,\kappa_n(z)))\\\notag
&\quad+\int_r^t\int_{\mathcal O}\Delta_nG^n_{t-s}(x,y)f^\prime(u^n(s,\kappa_n(y)))D_{r,z}u^n(s,\kappa_n(y))\ud y\ud s\\
&\quad+\int_r^t\int_{\mathcal O}G^n_{t-s}(x,y)\sigma^\prime(u^n(s,\kappa_n(y)))D_{r,z}u^n(s,\kappa_n(y))W(\ud s,\ud y),
\end{align}
if $r\le t$, and $D_{r,z}u^n(t,x)=0$, if $r>t$. 
Combining \eqref{Drzchintx} and \eqref{Drzchi}, we write
\begin{align}\label{Drzchi-Drzchi}
&D_{r,z}u^n(t,x)-D_{r,z}u(t,x):=I^n_{t,x}(r,z)+J^n_{t,x}(r,z)+K^n_{t,x}(r,z),
\end{align}
where for $r>t$, $I^n_{t,x}(r,z)=J^n_{t,x}(r,z)=K^n_{t,x}(r,z)=0$, and for $r\le t$, 
\begin{align*}
I^n_{t,x}(r,z)&=G^n_{t-r}(x,z)\sigma(u^n(r,\kappa_n(z)))-G_{t-r}(x,z)\sigma(u(r,z)),\\
J^n_{t,x}(r,z)
&=\int_{r}^t\int_{\mathcal O}\Delta_nG^n_{t-s}(x,y)f^\prime(u^n(s,\kappa_n(y)))\left[D_{r,z}u^n(s,\kappa_n(y))-D_{r,z}u(s,\kappa_n(y))\right]\ud y\ud s\\
&\quad+\int_{r}^t\int_{\mathcal O}\Delta_nG^n_{t-s}(x,y)f^\prime(u^n(s,\kappa_n(y)))\left[D_{r,z}u(s,\kappa_n(y))-D_{r,z}u(s,y)\right]\ud y\ud s\\
&\quad+\int_{r}^t\int_{\mathcal O}\Delta_nG^n_{t-s}(x,y)\left[f^\prime(u^n(s,\kappa_n(y)))-f^\prime(u(s,y))\right]D_{r,z}u(s,y)\ud y\ud s\\
&\quad+\int_{r}^t\int_{\mathcal O}\left[\Delta_nG^n_{t-s}(x,y)-\Delta G_{t-s}(x,y)\right]f^\prime(u(s,y))D_{r,z}u(s,y)\ud y\ud s\\
&=:J^{n,1}_{t,x}(r,z)+J^{n,2}_{t,x}(r,z)+J^{n,3}_{t,x}(r,z)+J^{n,4}_{t,x}(r,z),\\
K^n_{t,x}(r,z)
&=\int_{r}^t\int_{\mathcal O}G^n_{t-s}(x,y)\sigma^\prime(u^n(s,\kappa_n(y)))[D_{r,z}u^n(s,\kappa_n(y))-D_{r,z}u(s,\kappa_n(y))]W(\ud s,\ud y)\\
&\quad+\int_{r}^t\int_{\mathcal O}G^n_{t-s}(x,y)\sigma^\prime(u^n(s,\kappa_n(y)))\left[D_{r,z}u(s,\kappa_n(y))-D_{r,z}u(s,y)\right]W(\ud s,\ud y)\\
&\quad+\int_{r}^t\int_{\mathcal O}\left[G^n_{t-s}(x,y)\sigma^\prime(u^n(s,\kappa_n(y)))-G_{t-s}(x,y)\sigma^\prime(u(s,y))\right]D_{r,z}u(s,y)W(\ud s,\ud y)\\
&=:K^{n,1}_{t,x}(r,z)+K^{n,2}_{t,x}(r,z)+K^{n,3}_{t,x}(r,z).
\end{align*}
When $r>t$, we always set $J^{n,i}_{t,x}(r,z)=K^{n,j}_{t,x}(r,z)=0$ for $i=1,2,3,4$ and $j=1,2,3$.
Hereafter, let $\epsilon\ll 1$ be an arbitrarily fixed positive number. A combination of Lemma \ref{Holder-exact} and Theorem \ref{eq.strong-trun} reveals that for any $p\ge2$,
\begin{align}\label{eq.chi-chi}\notag
\|u^n(s,\kappa_n(y))-u(s,y)\|_p&\le \|u^n(s,\kappa_n(y))-u(s,\kappa_n(y))\|_p+\|u(s,\kappa_n(y))-u(s,y)\|_p\\
&\le Cn^{-1},
\end{align}
for all $(s,y)\in[0,T]\times\OO$. 
Then the Lipschitz continuity of $f^\prime$,
the Minkowski and Cauchy-Schwarz inequalities and \eqref{GnGn} produce
\begin{align*}
\|J^{n,3}_{t,x}\|_{L^2(\Omega;\cH)}&\!\le\! C_{\epsilon}\!\int_{0}^t\int_{\mathcal O}(t\!-\!s)^{-\frac{3}{4}-\epsilon}
\|u^n(s,\kappa_n(y))\!-\!u(s,y)\|_4\|Du(s,y)\|_{L^4(\Omega;\cH)}\ud y\ud s\\
&\le Cn^{-1}\sup_{t,x}\|Du(t,x)\|_{L^4(\Omega;\cH)}\le Cn^{-1},
\end{align*}
where \eqref{eq.chi-chi} and Lemma \ref{lem:chimp0} were used in the second line. Similarly, by the boundedness of $f^\prime$, Lemma \ref{Holder-e},  Lemma \ref{lem:chimp0}, and  \eqref{DeltaGnG},
\begin{align*}
\|J^{n,4}_{t,x}\|_{L^2(\Omega;\cH)}&\le \int_{0}^t\int_{\mathcal O}|\Delta_nG^n_{t-s}(x,y)-\Delta G_{t-s}(x,y)|\|f^\prime(u(s,y))Du(s,y)\|_{L^2(\Omega;\cH)}\ud y\ud s\\
&\le C \int_{0}^t\int_{\mathcal O}|\Delta_nG^n_{t-s}(x,y)-\Delta G_{t-s}(x,y)|
\ud y\ud s\le Cn^{-1}.
\end{align*}
Since $\sigma$ is bounded and Lipschitz continuous, it follows from the Minkowski inequality, \eqref{eq.chi-chi}, \eqref{Gn-G}, and \eqref{Gtxy} that for $p\ge2$,
\begin{align*}
\|I_{t,x}^n\|_{L^p(\Omega;\cH)}^2&\le \int_0^t\int_\OO \|G^n_{t-r}(x,z)\sigma(u^n(r,\kappa_n(z)))-G_{t-r}(x,z)\sigma(u(r,z))\|_{p}^2\ud z\ud r\\
&\le 2\int_0^t\int_\OO|G^n_{t-r}(x,z)-G_{t-r}(x,z)|^2 \|\sigma(u^n(r,\kappa_n(z)))\|_{p}^2\ud z\ud r\\
&\quad+2\int_0^t\int_\OO |G_{t-r}(x,z)|^2\|\sigma(u^n(r,\kappa_n(z)))-\sigma(u(r,z))\|_p^2\ud z\ud r\le Cn^{-2}.
\end{align*}
Replacing $\sigma$ by $\sigma^\prime$ in the above inequality, we also have
\begin{align*}
\int_0^t\int_\OO \|G^n_{t-s}(x,y)\sigma^\prime(u^n(s,\kappa_n(y)))-G_{t-s}(x,y)\sigma^\prime(u(s,y))\|_{4}^2\ud y\ud s\le Cn^{-2},
\end{align*}
which along with the Burkholder inequality for Hilbert space valued martingales (see e.g. \cite[(4.18)]{BP98}), the H\"older inequality and Lemma \ref{lem:chimp0} indicates
\begin{align*}
\|K^{n,3}_{t,x}\|_{L^2(\Omega;\cH)}^2&\le C\int_0^t\int_\OO \left\|\left(G^n_{t-s}(x,y)\sigma^\prime(u^n(s,\kappa_n(y)))-G_{t-s}(x,y)\sigma^\prime(u(s,y))\right)Du(s,y)\right\|^2_{L^2(\Omega;\cH)}\ud y\ud s\\
&\le Cn^{-2}\sup_{t,x}\|Du(t,x)\|_{L^4(\Omega;\cH)}^2\le Cn^{-2}.
\end{align*}
In order to estimate $I^{n,2}_{t,x}$ and $K^{n,2}_{t,x}$, we claim that
for  $p\ge2$, there exists some constant $C=C(p,T)$ such that for any $x_1,x_2\in\OO$ and $t\in(0,T]$,
\begin{align}\label{eq.Dchi-holder}
\|Du(t,x_1)-Du(t,x_2)\|_{L^p(\Omega;\cH)}\le C|x_1-x_2|.
\end{align}
Indeed, from \eqref{Drzchi}, we have that for $r\le t$,
\begin{align*}
D_{r,z}u(t,x_1)-D_{r,z}u(t,x_2)&=\left[G_{t-r}(x_1,z)-G_{t-r}(x_2,z)\right]\sigma(u(r,z))\\
&\quad+\int_r^t\int_{\mathcal O}\left[\Delta G_{t-s}(x_1,y)-\Delta G_{t-s}(x_2,y)\right]f^\prime(u(s,y))D_{r,z}u(s,y)\ud y\ud s\\
&\quad+\int_r^t\int_{\mathcal O}\left[G_{t-s}(x_1,y)-G_{t-s}(x_2,y)\right]\sigma^\prime(u(s,y))D_{r,z}u(s,y)W(\ud s,\ud y)\\
&=:L_1(r,z)+L_2(r,z)+L_3(r,z).
\end{align*}
For $r>t$, let $L_i(r,z)=0$, $i=1,2,3$.
The boundedness of $\sigma$ and Lemma \ref{Greg}
indicate
$$\|L_1\|^2_{L^p(\Omega;\cH)}\le C\int_0^t\int_\OO|G_{t-r}(x_1,z)-G_{t-r}(x_2,z)|^2\ud z\ud r\le C|x_1-x_2|^{2}.$$
Since $f^\prime$ is bounded, it follows from \eqref{DeltaGG}  and Lemma \ref{lem:chimp0} that 
\begin{align*}
\|L_2\|_{L^p(\Omega;\cH)}&\le C\int_0^t\int_{\mathcal O}\left|\Delta G_{t-s}(x_1,y)-\Delta G_{t-s}(x_2,y)\right|\|Du(s,y)\|_{L^p(\Omega;\cH)}\ud y\ud s\le C|x_1-x_2|.
\end{align*}
Similarly, it follows from the Burkholder inequality,  Lemmas \ref{Greg} and \ref{lem:chimp0} that 
$\|L_3\|_{L^p(\Omega;\cH)}\le C|x_1-x_2|.$
Gathering the above estimates of $L_1$, $L_2$ and $L_3$, we obtain 
\eqref{eq.Dchi-holder}.
By means of \eqref{eq.Dchi-holder} and \eqref{GnGn}, it can be verified that
$
\|J^{n,2}_{t,x}\|_{L^2(\Omega;\cH)}+\|K^{n,2}_{t,x}\|_{L^2(\Omega;\cH)}\le Cn^{-1}.
$
Substituting the above estimates of $\|I^{n}_{t,x}\|_{L^2(\Omega;\cH)}$, $\|J^{n,i}_{t,x}\|_{L^2(\Omega;\cH)}$, $i=2,3,4$, and $\|K^{n,j}_{t,x}\|_{L^2(\Omega;\cH)}$, $j=2,3$,
into \eqref{Drzchi-Drzchi}, we deduce that for  $0<\epsilon\ll1$,
\begin{align}\label{DunR-DuR}\notag
&\quad\|Du^n(t,x)-Du(t,x)\|^2_{L^2(\Omega;\cH)}\\\notag
&\le Cn^{-2}+C\left|\int_0^t\int_\OO|\Delta_nG^n_{t-s}(x,y)|\|Du^n(s,\kappa_n(y))-Du(s,\kappa_n(y))\|_{L^2(\Omega;\cH)}\ud y\ud s\right|^2\\\notag
&\quad+C\int_0^t\int_\OO|G^n_{t-s}(x,y)|^2\|Du^n(s,\kappa_n(y))-Du(s,\kappa_n(y))\|^2_{L^2(\Omega;\cH)}\ud y\ud s\\
&\le Cn^{-2}+C\int_0^t\int_\OO(t-s)^{-\frac{3}{4}-\epsilon}\|Du^n(s,\kappa_n(y))-Du(s,\kappa_n(y))\|^2_{L^2(\Omega;\cH)}\ud y\ud s,
\end{align}
where the last step used the H\"older inequality and \eqref{GnGn}.
Similar to \eqref{entx}, by taking the supremum over $x\in\OO$ on both sides of \eqref{DunR-DuR} and applying 
the Gronwall lemma with weak singularities (see e.g., \cite[Lemma 3.4]{GI98}), we complete the proof.
\end{proof}
\subsection{Convergence of density}
In this part, we present the convergence of density of the numerical solution $\{u^n(T,kh)\}_{n\ge2}$ for $k\in\Z_n$. 
In order to apply Proposition \ref{TVconvergence} with $X_\infty=u(T, x)$,
we impose Assumption \ref{A3} and investigate the negative moment estimate of  $Du(t, x)$.
\begin{Asp}\label{A3}
There exists some $\sigma_0>0$ such that $|\sigma(x)|>\sigma_0$, for any $x\in\R$.
\end{Asp}

\begin{lem}\label{lem:chimp}
Let $x\in\OO$ and Assumptions \ref{A1} and \ref{A3} hold. Then there is $\rho\in(0,1]$ such that
\begin{align}\label{nonde}
\E\big[\|Du(T,x)\|^{-2\rho}_{\cH}\big]\le C(\rho,T).
\end{align}
\end{lem}
\begin{proof}
To prove \eqref{nonde}, we need to use \cite[Proposition 3.2]{CH20}, which 
is summarized as follows: under Assumption \ref{A3},
if $x_i\in\OO$, $i=1,\ldots,d$, are distinct points, then for some $p_0>0$, there exists $\varepsilon_0=\varepsilon_0(p_0)$ such that for all $\varepsilon\in (0,\varepsilon_0)$, 
\begin{align}\label{C(t)}
\sup_{\xi\in\R^d,\|\xi\|=1}\mathbb P\left(\xi^\top \mathbb C(t)\xi\le \varepsilon\right)\le \varepsilon^{p_0},
\end{align}
where $\mathbb C(t):=(\langle Du(t,x_i),Du(t,x_j)\rangle_{\cH})_{1\le i,j\le d}$ denotes the Malliavin covariance matrix of $(u(t,x_1),\ldots,u(t,x_d))$ (the notation $u$ corresponds to $X_R$ in \cite{CH20}).

As a consequence of \eqref{C(t)} with $d=1$ and $t=T$, we have that for all $0<\varepsilon\le \varepsilon_0$, 
$\mathbb P(\|Du(T,x)\|^2_{\cH}\le\varepsilon)\le \varepsilon^{p_0},$
which implies that for any $ \rho< p_0$,
\begin{align*}
\sum_{n=1}^\infty n^{\rho-1}\mathbb P\left(\|Du(T,x)\|^{-2}_{\cH}\ge n\right)\le \sum_{n=1}^{\lfloor\varepsilon_0^{-1}\rfloor} n^{\rho-1}+\sum_{n= \lfloor\varepsilon_0^{-1}\rfloor+1}^\infty n^{\rho-1}n^{-p_0}\le C(\rho,\varepsilon_0).
\end{align*}
Then we have that for $0< \rho< \min\{p_0,1\}$ and $Z:=\|Du(T,x)\|^{-2}_{\cH}$,
\begin{align*}
\mathbb E\left[Z^{\rho}\right]&\le
1+\sum_{n=1}^\infty (n+1)^{\rho}\mathbb P(n\le Z< n+1)\le 2+\sum_{n=1}^\infty \left((n+1)^\rho-n^\rho\right)\mathbb P(Z\ge n)\\
&\le 2+\rho\sum_{n=1}^\infty n^{\rho-1}\mathbb P(Z\ge n)\le C(\rho,\varepsilon_0),
\end{align*}
which implies \eqref{nonde}. The proof is completed. 
\end{proof}

We are ready to give the main result of this section, which states that for $k\in\Z_n$, the density of the numerical solution $u^n(T,kh)$ exists and converges in $L^1(\R)$ to the density of the exact solution. 
The readers are referred to \cite{CC01,CH20} 
for the existence of the density $p_{u(t,x)}$ of the exact solution $u(t,x)$ for any $(t,x)\in[0,T]\times\OO$.

\begin{tho}\label{dTVR}
Suppose that Assumptions \ref{A1}, \ref{A4}, and \ref{A3} hold for $k=2$, and $u_0\in\C^3(\OO)$. Then for any $k\in\Z_n$, $u^n(T,kh)$ admits a density $p_{u^n(T,kh)}$, and moreover 
$$\lim_{n\rightarrow \infty}\|p_{u(T,kh)}-p_{u^n(T,kh)}\|_{L^1(\R)}=0.$$
\end{tho}

\begin{proof}
 By \cite[Theorem 2.3.3]{DN06} and \eqref{NCH}, we obtain that under Assumption \ref{A3},  for any $t\in(0,T]$, the law of $U(t)$ is absolutely continuous with respect to the Lebesgue measure on $\R^{n-1}$. Thus, $\{u^n(T,kh)\}_{k\in\Z_n}$ admits a density. Theorem \ref{eq.strong-trun}, Lemma \ref{lem:chimp0}, Proposition \ref{eq.strong-trun-D}, and
 Lemma \ref{lem:chimp} indicate that the conditions of 
Proposition \ref{TVconvergence}  are fulfilled for $\alpha=2\rho$, $X_n=u^n(T,kh)$ and $X_\infty=u(T,kh)$. As a result,
$$\lim_{n\rightarrow \infty}\mathrm{d}_{\mathrm{TV}}(u(T,kh),u^n(T,kh))=0,$$
which together to \eqref{pX-PX} completes the proof.
\end{proof}
\section{Full discretization}\label{S5}

For the purpose of effective computation, we combine the spatial FDM with a temporal exponential Euler method to obtain the full discretization of Eq.\ \eqref{CH}, and give the strong convergence rate of the fully discrete numerical solution in this section.

Let $\{t_i=i\tau,i=0,1,\ldots,m\}$ (with $m\ge2$) be a uniform partition of $[0,T]$, where $\tau:=T/m$  is the uniform time stepsize. Denote by
$\eta_m(s)=\tau\lfloor s/\tau\rfloor$ the largest time grid point smaller than $s$. 
By replacing $s$ in \eqref{unR} by $\eta_m(s)$, we obtain the full discretization $u^{m,n}=\{u^{m,n}(t,x); (t,x)\in[0,T]\times\OO\}$ given by
\begin{align}\label{unRF}\notag
u^{m,n}(t,x)&=\int_{\mathcal O}G^n_{t}(x,y)u_0(\kappa_n(y))\ud y\\\notag
&\quad+\int_{0}^{t}\int_{\mathcal O}\Delta_nG^n_{t-\eta_m(s)}(x,y)f(u^{m,n}(\eta_m(s),\kappa_n(y)))\ud y\ud s\\
&\quad+\int_{0}^{t}\int_{\mathcal O}G^n_{t-\eta_m(s)}(x,y)\sigma(u^{m,n}(\eta_m(s),\kappa_n(y)))W(\ud s,\ud y).
\end{align}
The discrete Green function $G^n$ satisfies the following estimates. 
\begin{lem}\label{Gn-regularity}
Let $\gamma\in(0,\frac{3}{8})$. Then for any $x,y\in\mathcal O$ and $s,t\in[0,T]$ with $s<t$,
\begin{gather*}
\int_0^t\int_{\mathcal O}|G^n_{t-r}(x,z)-G^n_{t-r}(y,z)|^2\ud z\ud r\le C|x-y|^2,\\
\int_0^s\int_{\mathcal O}|G^n_{t-r}(x,z)-G^n_{s-r}(x,z)|^2\ud z\ud r+\int_s^t\int_{\mathcal O}|G^n_{t-r}(x,z)|^2\ud z\ud r\le C_\gamma|t-s|^{2\gamma}.
\end{gather*}
\end{lem}
\begin{proof}
The proof is similar to that of \cite[Lemma 1.8]{CC01}.
It can be verified that
\begin{equation}\label{varphixy}
|\phi_{j,n}(x)-\phi_{j,n}(y)|\le \sqrt{2\pi^{-1}}j|x-y|.
\end{equation}
A combination of \eqref{orth} and \eqref{varphixy} implies
\begin{align*}
\int_0^t\int_{\mathcal O}|G^n_{t-r}(x,z)-G^n_{t-r}(y,z)|^2\ud z\ud r\le&\sum_{j=1}^{n-1}\frac{1}{2\lambda_{j,n}^2}|\phi_{j,n}(x)-\phi_{j,n}(y)|^2\le C|x-y|^2.
\end{align*}
By the uniform boundedness of $\phi_{j,n}$ and \eqref{orth},
\begin{align*}
\int_{\mathcal O}|G^n_{t}(x,z)|^2\ud z=\sum_{j=1}^{n-1}\exp(-2\lambda_{j,n}^2 t)|\phi_{j,n}(x)|^2
\le Ct^{-\frac{1}{4}}\int_0^\infty e^{-z^4}\ud z\le Ct^{-\frac{1}{4}},
\end{align*}
which indicates
$\int_s^t\int_{\mathcal O}|G^n_{t-r}(x,z)|^2\ud z\ud r\le C|t-s|^{\frac 34}$. 
Using \eqref{orth} and \eqref{1-ex}, we obtain
\begin{align}\label{Gnt-Gns}\notag
\int_0^s\int_{\mathcal O}|G^n_{t-r}(x,z)-G^n_{s-r}(x,z)|^2\ud z\ud r
&= \sum_{j=1}^{n-1}|1-\exp(-\lambda_{j,n}^2(t-s))|^2\frac{1-\exp(-2\lambda_{j,n}^2s)}{2\lambda_{j,n}^2}|\phi_{j,n}(x)|^2\\
&\le C\sum_{j=1}^{n-1}\lambda_{j,n}^{4\gamma-2}(t-s)^{2\gamma}\le C(t-s)^{2\gamma},
\end{align}
where $\gamma<\frac{3}{8}$.
The proof is completed.
\end{proof}

In a similar way, one can prove the following lemma. 
\begin{lem}\label{DeltaGn-regularity}
Let $\alpha\in(0,1)$. Then for any $x,y\in\mathcal O$ and $s,t\in[0,T]$ with $s<t$,
\begin{gather}\label{Gn1-1}
\int_s^t\int_{\mathcal O}|\Delta_n G^n_{t-r}(x,z)|\ud z\ud r\le C|t-s|^{\frac{3\alpha}{8}},\\\label{Gn2-1}
\int_0^s\int_{\mathcal O}|\Delta_nG^n_{t-r}(x,z)-\Delta_nG^n_{s-r}(y,z)|\ud z\ud r\le C(\alpha)(|x-y|+|t-s|^{\frac{3\alpha}{8}}).
\end{gather}
\end{lem}
\begin{proof}
By using the orthogonality of $\{\phi_j\circ\kappa_n\}_{j\in\mathbb Z_n}$ and the Cauchy--Schwarz inequality, \eqref{Gn1-1} and \eqref{Gn2-1} with $x=y$ and \eqref{Gn2-1} with $t=s$ can be obtained similarly as in \eqref{eq.Green}, \eqref{eq.Green1} and \eqref{DeltaGG}, respectively.
\end{proof}

\begin{prop}\label{Holder-un}
Let Assumption \ref{A1} hold and $u_0\in\mathcal C^2(\mathcal O)$. Then for any $\alpha\in(0,1)$ and $p\ge1$, there exists $C=C(p,T,\alpha)$ such that for any $(t,x)\in[0,T]\times\OO$,
\begin{align*} 
\|u^n(t,x)-u^n(s,y)\|_{p}\le C(|t-s|^{\frac{3\alpha}{8}}+C|x-y|).
\end{align*}
\end{prop}
\begin{proof}
As a result of Theorem \ref{eq.strong-trun} and Lemma \ref{Holder-e}, for $n\ge 2$ and $(t,x)\in[0,T]\times\OO$,
\begin{align}\label{eq.unbound}
\|u^n(t,x)\|_{p}\le \|u^n(t,x)-u(t,x)\|_{p}+\|u(t,x)\|_{p}
\le C(p,T,K).
\end{align}
Recall that $\tilde u^n_1(t,x):=\int_{\mathcal O}G^n_t(x,y)u_0(\kappa_n(y))\ud y$. It follows from \eqref{eq.u1tilde} and \eqref{un0Holder} that 
\begin{align*}
|\tilde u^n_1(t,x)-\tilde u^n_1(t,y)|&\le C|x-y|+
\int_0^t\int_\OO |\Delta_n G^n_r(x,z)-\Delta_n G^n_r(y,z)||\Delta_n u_0(z)|\ud z\ud r.
\end{align*}
Since $u_0\in \mathcal C^2(\OO)$, $|\Delta_n u_0(z)|\le C$ for $z\in\OO$, and hence \eqref{Gn2-1} implies $|\tilde u^n_1(t,x)-\tilde u^n_1(t,y)|\le C|x-y|.$
Similarly, based on \eqref{eq.u1tilde}, \eqref{Gn1-1}, and \eqref{Gn2-1}, one also has that
for any $\alpha\in(0,1)$, $|\tilde u^n_1(t,x)-\tilde u^n_1(s,x)|\le C|t-s|^{\frac{3\alpha}{8}}.$
Based on a standard argument as in the proof of Lemma \ref{Holder-exact}, 
it follows from \eqref{eq.unbound} and Lemmas \ref{Gn-regularity} and \ref{DeltaGn-regularity} that
$\|u^n(t,x)-\tilde u^n_1(t,x)-u^n(s,y)+\tilde u^n_1(s,y)\|_p\le C(|t-s|^{\frac{3\alpha}{8}}+C|x-y|)$. The proof is completed.
\end{proof}

\begin{tho}\label{Err} 
Suppose that Assumption \ref{A1} holds and $u_0\in\C^3(\OO)$. Then for every $p\ge1$ and $0<\epsilon\ll 1$, there exists some constant $C=C(p,T,K,\epsilon)$ such that for any $(t,x)\in[0,T]\times\OO$,
\begin{align*}
\|u^{m,n}(t,x)-u(t,x)\|_{p}
\le C(n^{-1}+m^{-\frac{3-\epsilon}{8}}).
\end{align*}
\end{tho}
\begin{proof}
Let $(t,x)\in[0,T]\times\OO$ and $0<\epsilon\ll 1$.
By virtue of Theorem \ref{eq.strong-trun}, it remains to show 
\begin{align*}
\|u^{m,n}(t,x)-u^n(t,x)\|_{p}\le C m^{-\frac{3-\epsilon}{8}}.
\end{align*}
By \eqref{unR}, \eqref{unRF}, the Minkowski inequality, the Burkholder inequality, and the Lipschitz continuity of $\sigma$ and $b$, we obtain that for any $p\ge 2$,
\begin{align*}
& \|u^{m,n}(t,x)-u^n(t,x)\|^2_{p}\le CH_{1}^{m,n}+CH_2^{m,n}+CQ_{1}^{m,n}+CQ_2^{m,n}\\
&+C\int_{0}^{t}\int_{\mathcal O}|\Delta_nG^n_{t-\eta_m(s)}(x,y)|\|u^{m,n}(\eta_m(s),\kappa_n(y))-u^{n}(\eta_m(s),\kappa_n(y))\|^2_{p}\ud y\ud s\\
&+C\int_{0}^{t}\int_{\mathcal O}|G^n_{t-\eta_m(s)}(x,y)|^2\|u^{m,n}(\eta_m(s),\kappa_n(y))-u^{n}(\eta_m(s),\kappa_n(y))\|_{p}^2\ud y\ud s,
\end{align*}
where 
\begin{align*}
H_{1}^{m,n}:&=\Big|\int_{0}^{t}\int_{\mathcal O}\big|\Delta_nG^n_{t-\eta_m(s)}(x,y)-\Delta_nG^n_{t-s}(x,y)\big|\|f(u^{n}(\eta_m(s),\kappa_n(y)))\|_{p}\ud y\ud s\Big|^2,\\
H_{2}^{m,n}:&=\Big|\int_{0}^{t}\int_{\mathcal O}|\Delta_nG^n_{t-s}(x,y)|\|f(u^{n}(\eta_m(s),\kappa_n(y)))-f(u^{n}(s,\kappa_n(y)))\|_{p}\ud y\ud s\Big|^2,\\
Q_{1}^{m,n}:&=\int_{0}^{t}\int_{\mathcal O}|G^n_{t-\eta_m(s)}(x,y)-G^n_{t-s}(x,y)|^2\|\sigma(u^{n}(\eta_m(s),\kappa_n(y)))\|_{p}^2\ud y\ud s,\\
Q_2^{m,n}:&=\int_{0}^{t}\int_{\mathcal O}|G^n_{t-s}(x,y)|^2\|\sigma(u^{n}(\eta_m(s),\kappa_n(y)))-\sigma(u^{n}(s,\kappa_n(y)))\|_{p}^2\ud y\ud s.
\end{align*}
Taking advantage of Corollary \ref{Holder-un} and \eqref{GnGn}, we obtain
$$H_{2}^{m,n}+Q_{2}^{m,n}\le C_\epsilon\sup_{t}|\eta_m(t)-t|^{\frac{3-\epsilon}{4}}\le C_\epsilon m^{-\frac{3-\epsilon}{4}}.$$
Similar to the proof of \eqref{Gnt-Gns}, one has that for $\alpha<\frac{3}{8}$,
\begin{align*}
\int_{0}^{t}\int_{\mathcal O}|G^n_{t-\eta_m(s)}(x,y)-G^n_{t-s}(x,y)|^2\ud y\ud s\le C_\epsilon \sup_t|\eta_m(t)-t|^{2\alpha},
\end{align*}
which along with the boundedness of $\sigma$ shows that $Q_{1}^{m,n}\le Cm^{-\frac{3(1-\epsilon)}{4}}$.
Similar to \eqref{Gn2-1} with $x=y$, we also have that for $\alpha<\frac{3}{8}$,
\begin{align*}
\int_{0}^{t}\int_{\mathcal O}\big|\Delta_nG^n_{t-\eta_m(s)}(x,y)-\Delta_nG^n_{t-s}(x,y)\big|\ud y\ud s
\le C\sup_t|\eta_m(t)-t|^{\alpha},
\end{align*}
which along with \eqref{eq.unbound} reveals that $H_{1}^{m,n}\le C_\epsilon m^{-\frac{3-\epsilon}{4}}$. 
Gathering the above estimates together yields that for any $t\in[0,T]$,
\begin{align*}
	&\quad\sup_{x}\|u^{m,n}(t,x)-u^n(t,x)\|^2_{p}\\
	&\le C_{\epsilon}m^{-\frac{3-\epsilon}{4}}+C\int_{0}^{t}\int_{\mathcal O}|\Delta_nG^n_{t-\eta_m(s)}(x,y)|\sup_{x}\|u^{m,n}(\eta_m(s),x)-u^n(\eta_m(s),x)\|^2_{p}\ud y\ud s\\
	&\quad+C\int_{0}^{t}\int_{\mathcal O}|G^n_{t-\eta_m(s)}(x,y)|^2\sup_{x}\|u^{m,n}(\eta_m(s),x)-u^n(\eta_m(s),x)\|^2_{p}\ud y\ud s.
\end{align*}
Letting $a(t):=\sup_{x}\|u^{m,n}(t,x)-u^n(t,x)\|^2_{p}$, $t\in[0,T]$ and using \eqref{GnGn}, we obtain
\begin{align}\label{at}
	a(t)\le C_{\epsilon}m^{-\frac{3-\epsilon}{4}}+C\int_0^t(t-\eta_m(s))^{-\frac{3}{4}-\epsilon}a(\eta_m(s))\ud s,\quad t\in[0,T].
\end{align}
Hence, for any $k=1,2,\ldots,m$, 
\begin{align*}
	a(t_k)&\le C_{\epsilon}m^{-\frac{3-\epsilon}{4}}+C\int_0^{t_k}(t_k-\eta_m(s))^{-\frac{3}{4}-\epsilon}a(\eta_m(s))\ud s
	=C_{\epsilon}m^{-\frac{3-\epsilon}{4}}+C\tau\sum_{i=0}^{k-1}t_{k-i}^{-\frac{3}{4}-\epsilon}a(t_i),
\end{align*}
which together with the discrete Gronwall lemma (see e.g. \cite[Lemma A.4]{Kruse}) implies that $\sup_{0\le k\le m}a(t_k)\le C_{\epsilon}m^{-\frac{3-\epsilon}{4}}.$
Finally, taking \eqref{at} into account gives
\begin{align*}
	a(t)&\le C_{\epsilon}m^{-\frac{3-\epsilon}{4}}+C_{\epsilon}m^{-\frac{3-\epsilon}{4}}\int_0^t(t-\eta_m(s))^{-\frac{3}{4}-\epsilon}\ud s
	\le C_{\epsilon}m^{-\frac{3-\epsilon}{4}},\quad t\in[0,T].
\end{align*}
Thus the proof is complete.
\end{proof}
\begin{rem}
The application of the orthogonality of $\{\phi_j\circ\kappa_n\}_{j\in\mathbb Z_n}$  plays a key role to obtain the temporal convergence order nearly $\frac38$ of the full discretization in Theorem \ref{Err}.
 For example, if the left hand of \eqref{Gn1-1} is estimated in the following way  
\begin{align*}
&\quad\int_0^s\int_{\mathcal O}|\Delta_nG^n_{t-r}(x,z)-\Delta_nG^n_{s-r}(x,z)|\ud z\ud r\\
&\le C\int_0^s\sum_{j=0}^{n-1}|\lambda_{j,n}|\exp(-\lambda_{j,n}^2(s-r))[1-\exp(-\lambda_{j,n}^2(t-s))]
\ud r\\
&
\le C \sum_{j=1}^{n-1}\frac{\lambda_{j,n}^{2\alpha}(t-s)^{\alpha}}{-\lambda_{j,n}}
\le C \sum_{j=1}^{\infty}j^{4\alpha-2}(t-s)^{\alpha}\le C_\alpha(t-s)^{\alpha},
\end{align*}
with $\alpha\in(0,\frac{1}{4})$,
 then we can only obtain $\int_0^s\int_{\mathcal O}|\Delta_nG^n_{t-r}(x,z)-\Delta_nG^n_{s-r}(x,z)|\ud z\ud r\le C(\alpha)|t-s|^{\frac{1}{4}-\epsilon}$ with $0<\epsilon\ll 1$. As a result, the temporal H\"older continuity exponent of $u^n$ is only nearly ${\frac{1}{4}}$, which leads to that  the temporal convergence order of the exponential Euler method is only nearly $\frac{1}{4}$.
\end{rem}
Combining Theorem \ref{Err} with a localized argument, we show  an $L^p(\Omega;\R)$ convergence order localized on a set of arbitrarily large probability for Eq.\ \eqref{CH} with polynomial nonlinearity.
\begin{cor}
Suppose that Assumption \ref{A2} hold and $u_0\in\C^3(\OO)$. Then for any $R\ge1$, $0<\epsilon\ll 1$ and $p\ge1$, there exists $C=C(R,T,p,\epsilon)$ such that
\begin{align*}
\E\left[\mathbf1_{\Omega_R\cap\Omega_{R}^{m,n}}|u^{m,n}(t,x)-u(t,x)|^p\right]\le C\big(n^{-1}+m^{-\frac{3-\epsilon}{8}}\big).
\end{align*}
\end{cor}
\begin{proof}
For $R\ge1$, denote $\Omega_R:=\big\{\omega\in\Omega:\sup_{t,x}|u(t,x,\omega)|\le R\big\}.$ 
Set $f_R=K_Rf$ with $K_R$ defined by \eqref{KR}. 
Consider the  localized Cahn-Hilliard equation
\begin{align}\label{ur}
\partial_t u_R+\Delta^2u_R=\Delta f_R(u_R)+\sigma(u_R)\dot{W},\quad R\ge1
\end{align}
with $u_R(0,\cdot)=u_0$ and  DBCs. Then the local property of stochastic integrals shows $u=u_R$ (i.e., for any $(t,x)\in[0,T]\times\OO$, $u(t,x)=u_R(t,x)$) on $\Omega_R$ a.s. 
Consider the numerical solution $u_R^{m,n}$ of \eqref{ur} based on the FDM in space and the exponential Euler method in time, i.e.,
\begin{align}\label{unRn}\notag
u_R^{m,n}(t,x)=&\int_{\mathcal O}G^n_t(x,y)u_0(\kappa_n(y))\ud y+\int_0^t\int_{\mathcal O}\Delta_nG^n_{t-\eta_m(s)}(x,y)f_R(u_R^{m,n}(\eta_m(s),\kappa_n(y)))\ud y\ud s\\
&+\int_0^t\int_{\mathcal O}G^n_{t-\eta_m(s)}(x,y)\sigma(u_R^{m,n}(\eta_m(s),\kappa_n(y)))W(\ud s,\ud y),
\end{align}
for $n,m\ge2$ and $(t,x)\in[0,T]\times\OO$. Seting
$\Omega_{R}^{m,n}:=\big\{\omega\in\Omega: \sup_{t,x}|u^{m,n}(t,x,\omega)|\le R\big\}$, and comparing \eqref{unRn} with \eqref{unRF}, it follows from the local property of stochastic integrals that $u_R^{m,n}=u^{m,n}$ on $\Omega_{R}^{m,n}$ a.s.
For fixed $R\ge1$, since $f_R$ satisfies Assumption \ref{A1}, Theorem \ref{Err} indicates that there exists some constant $C=C(R,T,p,\epsilon)$ such that 
for $0<\epsilon\ll 1$ and $p\ge1$,
\begin{equation*}
\E\left[|u^{m,n}_R(t,x)-u_R(t,x)|^p\right]\le C\big(n^{-1}+m^{-\frac{3-\epsilon}{8}}\big).
\end{equation*}
Since $u^{m,n}(t,x)$ and $u(t,x)$ have almost surely continuous trajectories, we have $\lim_{R\rightarrow \infty}\mathbb P(\Omega_R)=\lim_{R\rightarrow \infty}\mathbb P(\Omega_R^{m,n})=1$, which implies $\lim_{R\rightarrow \infty}\mathbb P(\Omega_R\cap\Omega_{R}^{m,n})=1$.
By  $u=u_R$ on $\Omega_R$ a.s., and $u_R^{m,n}=u^{m,n}$ on $\Omega_{R}^{m,n}$ a.s., we obtain 
\begin{align*}
\E\left[\mathbf1_{\Omega_R\cap\Omega_{R}^{m,n}}|u^{m,n}(t,x)-u(t,x)|^p\right]\le
\E\left[|u^{m,n}_R(t,x)-u_R(t,x)|^p\right]\le C\big(n^{-1}+m^{-\frac{3-\epsilon}{8}}\big).
\end{align*}
The proof is completed.
\end{proof}

\begin{rem}
Theorems \ref{eq.strong-trun} and \ref{dTVR} indicate that when applying the spatial FDM to the localized Cahn--Hilliard equation \eqref{LCH}, the associated numerical solution is strongly convergent and the density of the numerical solution converges in $L^1(\R)$. In addition, Section \ref{S2} gives the uniform moment estimate and H\"older continuity of the exact solution for Eq.\ \eqref{CH} with $f$ being a polynomial of degree $3$ with a positive dominant coefficient. We expect to combine the above results with the localization technique to study the strong convergence of the spatial FDM and the density convergence of the associated numerical solution for the stochastic Cahn–Hilliard equation with polynomial nonlinearity and multiplicative noise in the future.

\end{rem}

\bibliographystyle{plain}
\bibliography{mybibfile}

\begin{thebibliography}{10}

\bibitem{ACQ20}
Rikard Anton, David Cohen, and Lluis Quer-Sardanyons.
\newblock A fully discrete approximation of the one-dimensional stochastic heat
  equation.
\newblock {\em IMA J. Numer. Anal.}, 40(1):247--284, 2020.

\bibitem{BP98}
Vlad Bally and Etienne Pardoux.
\newblock Malliavin calculus for white noise driven parabolic {SPDE}s.
\newblock {\em Potential Anal.}, 9(1):27--64, 1998.

\bibitem{BT96}
Vlad Bally and Denis Talay.
\newblock The law of the {E}uler scheme for stochastic differential equations.
  {II}. {C}onvergence rate of the density.
\newblock {\em Monte Carlo Methods Appl.}, 2(2):93--128, 1996.

\bibitem{CC01}
Caroline Cardon-Weber.
\newblock Cahn-{H}illiard stochastic equation: existence of the solution and of
  its density.
\newblock {\em Bernoulli}, 7(5):777--816, 2001.

\bibitem{CCZZ18}
Shimin Chai, Yanzhao Cao, Yongkui Zou, and Wenju Zhao.
\newblock Conforming finite element methods for the stochastic
  {C}ahn-{H}illiard-{C}ook equation.
\newblock {\em Appl. Numer. Math.}, 124:44--56, 2018.

\bibitem{CCHS20}
Chuchu {Chen}, Jianbo {Cui}, Jialin {Hong}, and Derui {Sheng}.
\newblock {Convergence of Density Approximations for Stochastic Heat Equation}.
\newblock {\em arXiv:2007.12960}.

\bibitem{CQ16}
David Cohen and Llu\'{\i}s Quer-Sardanyons.
\newblock A fully discrete approximation of the one-dimensional stochastic wave
  equation.
\newblock {\em IMA J. Numer. Anal.}, 36(1):400--420, 2016.

\bibitem{CH20}
Jianbo Cui and Jialin Hong.
\newblock Absolute continuity and numerical approximation of stochastic
  {C}ahn-{H}illiard equation with unbounded noise diffusion.
\newblock {\em J. Differential Equations}, 269(11):10143--10180, 2020.

\bibitem{CHS19}
Jianbo {Cui}, Jialin {Hong}, and Derui {Sheng}.
\newblock {Convergence in Density of Splitting AVF Scheme for Stochastic
  Langevin Equation}.
\newblock {\em arXiv:1906.03439}.

\bibitem{CHS21}
Jianbo Cui, Jialin Hong, and Liying Sun.
\newblock Strong convergence of full discretization for stochastic
  {C}ahn-{H}illiard equation driven by additive noise.
\newblock {\em SIAM J. Numer. Anal.}, 59(6):2866--2899, 2021.

\bibitem{DG01}
A.~M. Davie and J.~G. Gaines.
\newblock Convergence of numerical schemes for the solution of parabolic
  stochastic partial differential equations.
\newblock {\em Math. Comp.}, 70(233):121--134, 2001.

\bibitem{DN91}
Qiang Du and R.~A. Nicolaides.
\newblock Numerical analysis of a continuum model of phase transition.
\newblock {\em SIAM J. Numer. Anal.}, 28(5):1310--1322, 1991.

\bibitem{EL92}
Charles~M. Elliott and Stig Larsson.
\newblock Error estimates with smooth and nonsmooth data for a finite element
  method for the {C}ahn-{H}illiard equation.
\newblock {\em Math. Comp.}, 58(198):603--630, S33--S36, 1992.

\bibitem{FKLL18}
Daisuke Furihata, Mih\'{a}ly Kov\'{a}cs, Stig Larsson, and Fredrik Lindgren.
\newblock Strong convergence of a fully discrete finite element approximation
  of the stochastic {C}ahn-{H}illiard equation.
\newblock {\em SIAM J. Numer. Anal.}, 56(2):708--731, 2018.

\bibitem{GI98}
Istv\'{a}n Gy\"{o}ngy.
\newblock Lattice approximations for stochastic quasi-linear parabolic partial
  differential equations driven by space-time white noise. {I}.
\newblock {\em Potential Anal.}, 9(1):1--25, 1998.

\bibitem{HW96}
Y.~Hu and S.~Watanabe.
\newblock Donsker's delta functions and approximation of heat kernels by the
  time discretization methods.
\newblock {\em J. Math. Kyoto Univ.}, 36(3):499--518, 1996.

\bibitem{KD14}
Davar Khoshnevisan.
\newblock {\em Analysis of {S}tochastic {P}artial {D}ifferential {E}quations},
  volume 119 of {\em CBMS Regional Conference Series in Mathematics}.
\newblock Published for the Conference Board of the Mathematical Sciences,
  Washington, DC; by the American Mathematical Society, Providence, RI, 2014.

\bibitem{KHA97}
A.~Kohatsu-Higa.
\newblock High order {I}t\^{o}-{T}aylor approximations to heat kernels.
\newblock {\em J. Math. Kyoto Univ.}, 37(1):129--150, 1997.

\bibitem{KLM11}
Mih\'{a}ly Kov\'{a}cs, Stig Larsson, and Ali Mesforush.
\newblock Finite element approximation of the {C}ahn-{H}illiard-{C}ook
  equation.
\newblock {\em SIAM J. Numer. Anal.}, 49(6):2407--2429, 2011.

\bibitem{Kruse}
R.~Kruse.
\newblock {\em Strong and {W}eak {A}pproximation of {S}emilinear {S}tochastic
  {E}volution {E}quations}, volume 2093 of {\em Lecture Notes in Mathematics}.
\newblock Springer, Cham, 2014.

\bibitem{LM11}
Stig Larsson and Ali Mesforush.
\newblock Finite-element approximation of the linearized
  {C}ahn-{H}illiard-{C}ook equation.
\newblock {\em IMA J. Numer. Anal.}, 31(4):1315--1333, 2011.

\bibitem{NP13}
Ivan Nourdin and Guillaume Poly.
\newblock Convergence in total variation on {W}iener chaos.
\newblock {\em Stochastic Process. Appl.}, 123(2):651--674, 2013.

\bibitem{DN06}
David Nualart.
\newblock {\em The {M}alliavin {C}alculus and {R}elated {T}opics}.
\newblock Probability and its Applications (New York). Springer-Verlag, Berlin,
  second edition, 2006.

\bibitem{QW20}
Ruisheng Qi and Xiaojie Wang.
\newblock Error estimates of semidiscrete and fully discrete finite element
  methods for the {C}ahn-{H}illiard-{C}ook equation.
\newblock {\em SIAM J. Numer. Anal.}, 58(3):1613--1653, 2020.

\bibitem{QS04}
Llu\'{\i}s Quer-Sardanyons and Marta Sanz-Sol\'{e}.
\newblock A stochastic wave equation in dimension 3: smoothness of the law.
\newblock {\em Bernoulli}, 10(1):165--186, 2004.

\end{thebibliography}

\end{document}